\documentclass{article}%
\usepackage[utf8]{inputenc}%
\usepackage[english]{babel}%
\usepackage{graphicx}%
\usepackage{amsmath}%
\usepackage{caption}
\usepackage{amssymb}%
\usepackage{amsthm}%
\usepackage{amsfonts}%
\usepackage{hyperref}
\newtheorem{thm}{Theorem}[section]
\newtheorem{prop}[thm]{Proposition}
\newtheorem{df}[thm]{Definition}
\newtheorem{rmk}[thm]{Remark}
\newtheorem{corollary}[thm]{Corollary}
\newtheorem{lemma}[thm]{Lemma}
\newcommand{\bysame}{\leavevmode\hbox to3em{\hrulefill}\thinspace}
\newcommand{\mq}{\mathsf{q}}

\newcommand{\Fc}{\mathcal{F}}
\newcommand{\si}{\sigma}
\newcommand{\pd}{\partial}
\newcommand{\p}{\mathbf{p}}
\newcommand{\g}{\mathbf{g}}
\newcommand{\gh}{\mathbf{\hat g}}
\renewcommand{\phi}{\varphi}
\newcommand{\basf}{\varepsilon}
\newcommand{\mQ}{\mathcal{Q}}
\newcommand{\trp}{\theta}
\newcommand{\Rc}{\mathcal{R}}
\newcommand{\al}{\alpha}
\newcommand{\la}{\lambda}
\newcommand{\Sb}{\mathbb{S}}
\newcommand{\N}{\mathbb{Z}_{>0}}
\newcommand{\Nn}{\mathbb{Z}_{\ge0}}
\newcommand{\mi}{\mathrm{i}}
\newcommand{\mj}{\mathrm{j}}
\newcommand{\mc}{\mathrm{c}}
\newcommand{\X}[1]{\mathbf{X}_{#1}}
\newcommand{\sk}{\Omega_+}
\newcommand{\x}{\mathsf{x}}
\newcommand{\dhh}{\mathsf{h}}
\newcommand{\Df}{\mathsf{D}}
\begin{document}
\title{Random Walks on Strict Partitions}
\author{Leonid Petrov\thanks{Kharkevich Institute for Information Transmission Problems.}}
\date{}
\maketitle

\begin{abstract}
	We consider 
	a certain sequence of random walks.
	The state space of the $n$th random walk is the set of all strict partitions
	of $n$ (that is, partitions without equal parts).
	We prove that, as $n$ goes to infinity, 
	these random walks converge to a continuous-time 
	Markov process. 
	The state space of this process is the infinite-dimensional simplex
	consisting of all nonincreasing infinite sequences of nonnegative numbers
	with sum less than or equal to one.
	The main result about the limit process is the expression
	of its the pre-generator as a formal 
	second order differential operator 
	in a polynomial algebra. 
	
	Of separate interest is the generalization of Kerov interlacing coordinates 
	to the case of shifted Young diagrams.
\end{abstract}

\section{Introduction}

\subsection{The Schur graph and its boundary}
Let $\Sb_n$, $n\in\Nn$, denote the finite set 
consisting of all strict partitions (that is, partitions
without equal parts) of $n$.\footnote{The set $\Sb_0$ consists of
the empty partition $\varnothing$.} Strict partitions are represented by
shifted Young diagrams \cite[Ch. I, \S1, Example 9]{Macdonald1995}. 
The Schur graph 
is the graded graph consisting of 
all shifted Young diagrams (with edge multiplicities given 
by (\ref{mult}) below).

As can be proved exactly as in \cite{Kerov1998} 
using the results of \cite{IvanovNewYork3517-3530},\footnote{Another
proof can be found in the paper \cite{Nazarov1992}.}
the Martin boundary of the 
Schur graph can be identified with the infinite-dimensional ordered simplex 
\begin{equation*}
        \sk:=\left\{ \x=(\x_1,\x_2,\dots)\colon \x_1\ge\x_2\ge\dots\ge0,\ \sum_{i}\x_i\le1 \right\}.
\end{equation*}
The simplex $\sk$ viewed as a subspace
of the infinite-dimensional cube $\left[ 0,1 \right]^{\infty}$
(which in turn is equipped with the product topology)
is a compact, metrizable and separable space.

\subsection{Projective representations of symmetric groups}\label{s0.2}
Each $\Sb_n$, $n\ge1$, may be regarded as a projective dual 
object to the symmetric group $\mathfrak{S}_n$
in the sense that $\Sb_n$ parametrizes the 
irreducible projective representations of $\mathfrak{S}_n$ \cite{Hoffman1992,Schur1911}.
The simplex $\sk$ can be viewed as a kind of projective dual to 
the infinite symmetric group $\mathfrak{S}_\infty$.
That is, the points of $\sk$ parametrize the indecomposable normalized
projective characters of the 
group $\mathfrak{S}_\infty$ \cite{Nazarov1992}.

The theory of projective representations of symmetric groups is in many 
aspects similar to the theory of ordinary representations.
Let us indicate some of them:
\begin{itemize}
	\item The (ordinary) dual object to $\mathfrak{S}_n$
		is the set of all ordinary (i.e., not necessary strict)
		partitions of~$n$;
	\item The role of $\sk$ in the theory
		of ordinary representations 
		is played by the {\em{}Thoma simplex\/} $\Omega$ that
		consists of couples
		$(\omega;\omega')\in\left[ 0,1 \right]^{\infty}\times\left[ 0,1 \right]^{\infty}$ 
		satisfying the following conditions:
		\begin{equation*}
			\omega_1\ge\omega_2\ge\dots\ge0,\qquad
			\omega'_1\ge\omega'_2\ge\dots\ge0,\qquad
			\sum_{i}\omega_i+
			\sum_{j}\omega'_j\le1.
		\end{equation*}
		The points of $\Omega$ parametrize the indecomposable normalized
		characters of $\mathfrak{S}_\infty$ \cite{Thoma1964}. 
	\item The role that Schur's $\mQ$-functions play in the theory of projective
		representations \cite{Hoffman1992} is taken by ordinary Schur functions.
\end{itemize}

There is a natural embedding 
of $\sk$ into $\Omega$ introduced in 
\cite{GnedinIntern.Math.ResearchNotices2006Art.ID5196839pp.}.
This map sends $\x=(\x_1,\x_2,\dots)\in\sk$ to $(\omega,\omega')\in\Omega$
with $\omega=\omega'=(\x_1/2,\x_2/2,\dots)$.  
In \S\ref{s8.1} below we discuss this embedding in more detail.

\subsection{Multiplicative measures}\label{s0.3}
In \cite{Borodin1997} 
A.~Borodin introduced 
multiplicative 
coherent systems of measures
(or central measures in the sense of \cite{Kerov1990})
on the
Schur graph.
This is a sequence of probability measures $M_n^\al$ on $\Sb_n$, $n\in\Nn$, 
depending on one real parameter $\al\in(0,+\infty)$.
Below we call this object simply the {\em{}multiplicative measures\/}.

According to a general formalism (explained, e.g., in \cite{Kerov1998})
the multiplicative measures $\left\{ M_n^\al \right\}$ give rise to a one-parameter 
family of probability measures $\mathsf P^{(\al)}$ on $\sk$.
Namely, every set $\Sb_n$ can be embedded into $\sk$:
a strict partition $\la\in\Sb_n$ maps to a point $(\la_1/n,\la_2/n,\dots)\in\sk$, where $\la_i$
are the components of $\la$.
As $\al$ remains fixed and $n$ goes to infinity,
the images of $M_n^\al$ 
under these embeddings weakly converge to $\mathsf P^{(\al)}$.

\subsection{The model of random walks}
To a coherent system of measures on a graded graph
one can associate a sequence of random walks on the floors of the graph.\footnote{We
assume that a graded graph
satisfies some additional conditions
(listed, e.g., in \cite[\S3]{Fulman2007}) that allow to consider coherent systems on it.
In fact, the Schur graph satisfies them.
See also \cite[\S1]{Borodin2007}.}
These random walks are called the {\em{}up-down Markov chains\/}.

The up/down Markov chains first appeared in a paper by J.~Fulman~\cite{Fulman2005}.
He was interested in such questions as 
their eigenstructure, eigenvectrors and convergence rates.
In the papers \cite{Fulman2005,Fulman2007}
many examples of up/down Markov chains associated to various coherent systems
on various graphs are studied
including
the Plancherel measures and the $z$-measures on the Young
graph, the Ewens-Pitman's partition structures
on the Kingman graph,\footnote{The Young
graph is the graph consisting of all ordinary Young diagrams
(they are identified with ordinary partitions as in \cite[Ch. I, \S1]{Macdonald1995}).
The $z$-measures originated 
from the problem of harmonic analysis for the 
infinite symmetric group $\mathfrak{S}_\infty$ \cite{Kerov1993,Kerov2004} and were
studied in detail by 
A.~Borodin and G.~Olshanski, see the bibliography in \cite{Borodin2007}.

The set of vertices of the Kingman graph is the same as of the Young graph, but the edge 
multiplicities are different. The Ewens-Pitman's partition structure
(this is a special name for the coherent system of measures
on the Kingman graph, the term is 
due to J.~F.~C.~Kingman \cite{Kingman1978})
was introduced
in \cite{Ewens1979,Pitman1992}. It is closely related to the Poisson-Dirichlet
measure (see, e.g., \cite{Pitman2002} and bibliography therein).}
and the Plancherel
measures on the Schur graph.\footnote{We give the
definition of the Plancherel measures on the Schur graph in \S\ref{s1.4} below.}

In \cite{Borodin2007,Petrov2007,Olshanski2009} 
the limit behaviour of various up/down Markov chains 
is studied. The paper \cite{Borodin2007} deals with the limit behaviour 
of the chains associated to the 
$z$-measures on the Young graph. In \cite{Petrov2007}
the chains corresponding to the Ewens-Pitman's partition structure are studied, and
in \cite{Olshanski2009} the general case of the Young graph with Jack edge multiplicities
is considered.

In this paper we consider a sequence of up/down Markov
chains associated to the multiplicative 
measures on the Schur graph.
These chains depend on the parameter $\al$.
The $n$th chain lives on $\Sb_n$ and preserves the 
probability measure $M_n^\al$. We study the limit behaviour of these 
random walks as $n\to\infty$.

\subsection{The limit diffusion and its pre-generator}
Assume that $\al\in(0,+\infty)$ is fixed.
Let us embed each set $\Sb_n$ into $\sk$ as described above in \S\ref{s0.3},
and let the discrete
time of the $n$th up/down chain be scaled by the factor $n^{-2}$.
We show that under these space and time scalings
the up/down chains 
converge, as $n\to\infty$,
to a diffusion process\footnote{By a diffusion process
we mean a strong Markov process with continuous sample paths.} $\X{\al}(t)$, $t\ge0$,
in the simplex $\sk$. 
We also show that $\X{\al}(t)$ preserves the measure $\mathsf{P}^{(\al)}$, 
is reversible and ergodic with respect to it.

The main result of the present paper is the formula
for the pre-generator of the process $\X\al(t)$. To formulate the result we
need some notation.

By $C(\sk)$ denote the Banach algebra of real-valued continuous functions on $\sk$
with pointwise operations and the uniform norm.
Let $\Fc$ be a dense subspace of $C(\sk)$ freely 
generated (as a commutative unital algebra) by the
algebraically independent continuous functions
$\mq_{2k}(\x):=\sum_{i=1}^{\infty}\x_i^{2k+1}$, $k=1,2,\dots$.
Define an operator $A\colon\Fc\to\Fc$ depending on the parameter $\al$:
\begin{equation}\label{f0.1}
        \left.
        \begin{array}{l}
                \displaystyle
                A=\sum_{i,j=1}^{\infty}(2i+1)(2j+1)\left( \mq_{2i+2j}-\mq_{2i}\mq_{2j} \right)
                \frac{\partial^2}{\partial\mq_{2i}\partial\mq_{2j}}\\\displaystyle\qquad+
                2\sum_{i,j=0}^{\infty}\left( 2i+2j+3 \right)\mq_{2i}\mq_{2j}\frac{\partial}{\partial \mq_{2i+2j+2}}
                -\sum_{i=1}^{\infty}(2i+1)\left( 2i+\frac\al2 \right)
                \mq_{2i}\frac{\partial}{\partial\mq_{2i}},
        \end{array}
        \right.
\end{equation}
where, by agreement, $\mq_0=1$. This is a formal differential operator in the polynomial
algebra $\Fc=\mathbb{R}\left[ \mq_2,\mq_4,\mq_6,\dots \right]$.

We show that 
the operator $A$ is closable in $C(\sk)$ and that the process
$\X{\al}(t)$ is generated by the closure $\overline A$ of the operator $A$.

\subsection{The method}
The formulation of the results is given in probabilistic terms.
However, the use of probabilistic technique in the proofs
generally reduces to the application of certain results from the paper
\cite{Trotter1958} and the book \cite{Ethier1986} concerning approximations of continuous
semigroups by discrete ones.

The essential part of the paper consists of the 
computations in a polynomial algebra.
To obtain the formula (\ref{f0.1}) for the pre-generator
we use the methods similar to those of
\cite{Olshanski2009}.
This involves the restatement 
of some of the results concerning ordinary
Young diagrams to our situation. 
In particular, we introduce Kerov interlacing coordinates of shifted Young diagrams
which are similar to
interlacing coordinates of ordinary Young diagrams 
introduced and studied by S.~Kerov in \cite{Kerov2000}.
Kerov interlacing coordinates of shifted Young diagrams are of separate interest.

We also give an alternative 
expression 
for the pre-generator $A$. Namely,
we
compute the action of
$A$ on Schur's $\mQ$-functions
(Proposition \ref{p6.8} (2) below).
This is done in \S\ref{s2} exactly as in 
\cite[\S4]{Borodin2007} 
with ordinary Schur functions replaced by
Schur's $\mQ$-functions.
In this argument we use the
formula (\ref{f36}) for dimension of skew shifted 
Young diagrams which is due to V.~Ivanov \cite{IvanovNewYork3517-3530}.
Note that the formula (\ref{f86}) for the action of $A$ on Schur's $\mQ$-functions
is not formally necessary for the rest of the results of the present paper
(see Remark \ref{p90.1} below). 

\subsection{Organization of the paper}
In \S\ref{s1.1}--\ref{s1.3} we recall the definition of the Schur graph.
We also recall coherent systems
associated to this graph and the corresponding up/down Markov chains.
In \S\ref{s1.4} we recall the multiplicative measures
on the Schur graph
introduced by A.~Borodin \cite{Borodin1997}.
They depend on a parameter $\al\in(0,+\infty)$.\footnote{Note that
in \cite{Borodin1997} the parameter $\al$ is denoted by $x$.}

In \S\ref{s20} we introduce 
Kerov interlacing coordinates of shifted Young diagrams and
study their properties. Here we restate 
some of the results of the paper \cite{Kerov2000} and apply them
to our situation.

In \S\ref{s2.1} we consider the polynomial algebra $\Gamma$ 
generated by the odd Newton power sums. 
The basis for $\Gamma$ is formed by Schur's $\mQ$-functions $\mQ_\la$
(indexed by strict partitions).
In \S\ref{s2.2}--\ref{s2.3} we
prove a useful formula for the action of the $n$th up/down Markov chain transition operator
(corresponding to the multiplicative measures on the Schur graph)
on Schur's $\mQ$-functions. 
The argument here is the same as in \cite[\S4]{Borodin2007}.

In \S\ref{s3} we prove some facts
concerning the algebra $\Gamma$ that are used in \S\ref{s4}--\ref{s5}.

In \S\ref{s4}--\S\ref{s5} we compute the ``differential'' form of 
the $n$th up/down Markov chain
transition operator corresponding to the multiplicative measures
(see Theorem \ref{p5.1} below for the exact result).

In \S\ref{s6} we use the general results 
of the paper \cite{Trotter1958} and the book \cite{Ethier1986}
to prove the convergence, as $n\to\infty$, of our up/down Markov
chains to a continuous time Markov process $\X{\al}(t)$ in $\sk$.
We also prove the differential formula for the pre-generator of this process
and study some other properties
of $\X{\al}(t)$.

\subsection{Acknowledgement}
I am very grateful to Grigori Olshanski for the setting of the problem,
permanent attention and fruitful discussions, and to Vladimir Ivanov for helpful discussions.

\section{Multiplicative measures}\label{s1}
\subsection{The Schur graph}\label{s1.1}
A {\em{}partition\/} is an infinite non-increasing sequence of nonnegative integers
\begin{equation*}
	\la=(\la_1,\la_2,\dots,\la_{\ell(\la)},0,0,\dots),\qquad
	\la_1\ge\la_2\ge\dots\ge\la_{\ell(\la)}>0,\qquad \la_i\in\N,
\end{equation*}
having only finitely many nonzero members.
Their number $\ell(\la)\ge0$ is called the 
{\em{}length of the partition\/}.
The {\em{}weight of the partition\/}
is
$|\la|:=\sum_{i=1}^{\ell(\la)}\la_i$.
A partition $\la$ is called {\em{}strict\/} if
it does not contain similar terms:
$\la_1>\la_2>\dots>\la_{\ell(\la)}>0$.
We denote strict partitions by $\la,\mu,\nu,\varkappa,\dots$,
and ordinary (i.e., not necessary strict) partitions 
by $\sigma,\rho,\tau,\dots$.

As explained in \cite[Ch. I, \S1, Example 9]{Macdonald1995}, to every strict partition corresponds a {\em{}shifted Young diagram\/}.
The shifted Young diagram of the form 
$\la$ consists of $\ell(\la)$ rows.
Each $i$th row ($i=1,\dots,\ell(\la)$) has $\la_i$ boxes,
and for $j=1,\dots,\ell(\la)-1$ the first box of the $(j+1)$th 
row is right under the second box of the $j$th row.
We identify strict partitions and corresponding shifted Young diagrams.
For example,
Figure \ref{fig1}
shows the shifted Young diagram of the form $(6,5,3,1)$.
\begin{figure}[htpb]
	\begin{center}
		\includegraphics{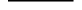}
	\end{center}
	\caption{Figure \ref{fig1}.}
	\label{fig1}
\end{figure}

If $\la$ and $\mu$ are shifted Young diagrams and $\la$ is obtained from $\mu$
by adding one box, then we write $\la\searrow\mu$ (or, equivalently, $\mu\nearrow\la$).
Denote this box (that distinguishes $\la$ and $\mu$) by $\la/\mu$.

For two shifted Young diagrams $\mu$ and $\la$ 
such that $|\la|=|\mu|+1$ we set
\begin{equation}\label{mult}
	\kappa(\mu,\la):=\left\{
	\begin{array}{ll}
		2,&\mbox{if $\mu\nearrow\la$ and $\ell(\la)=\ell(\mu)$};\\
		1,&\mbox{if $\mu\nearrow\la$ and $\ell(\la)=\ell(\mu)+1$};\\
		0,&\mbox{otherwise}.
	\end{array}
	\right.
\end{equation}

All shifted Young diagrams are organized in a graded set 
$\Sb=\bigsqcup_{n=0}^{\infty}\Sb_n$, where
$\Sb_n:=\left\{ \la\colon|\la|=n \right\}$, $n\in\N$, and
$\Sb_0:=\left\{ \varnothing \right\}$.
This set is equipped with the structure of a graded
graph. It has edges only between consecutive floors $\Sb_n$ and $\Sb_{n+1}$, $n\in\Nn$.
If $\mu\in\Sb_n$ and $\la\in\Sb_{n+1}$, then we draw $\kappa(\mu,\la)$
edges between $\mu$ and $\la$. Let edges be oriented in the direction from $\Sb_n$ to $\Sb_{n+1}$.
We call this oriented graded graph 
the {\em{}Schur graph\/}.\footnote{Sometimes (e.g., in \cite{Borodin1997})
the same graph with simple edges
is called the Schur graph.
These two graphs have the same down transition functions (see \S\ref{s1.3} below),
hence for us the difference between them is inessential.}

By $\dhh(\mu,\la)$ denote the total number 
of (oriented) paths from $\mu$ to $\la$ 
in the graph $\Sb$.
Clearly, $\dhh(\mu,\la)$ vanishes unless
$\mu\subset\la$ (that is, unless the shifted Young diagram
$\mu$ is a subset of the shifted Young diagram $\la$).
Set $\dhh(\la):=\dhh(\varnothing,\la)$.
This function has the form \cite[Ch. III, \S8, Example 12]{Macdonald1995}:\footnote{The factor
$2^{|\la|-\ell(\la)}$ (that does not enter the corresponding formula
in \cite{Macdonald1995}) appears due to the edge multiplicities (\ref{mult})
in our version of the Schur graph.}
\begin{equation}\label{f1}
	\dhh(\la)=2^{|\la|-\ell(\la)}\cdot\frac{|\la|!}{\la_1!\la_2!\dots\la_{\ell(\la)}!}
	\prod_{1\le i<j\le\ell(\la)}\frac{\la_i-\la_j}{\la_i+\la_j},\qquad\la\in\Sb.
\end{equation}
Note that if $\la$ is not strict, then this formula reduces to $\dhh(\la)=0$.
There is also an explicit
formula for the function $\dhh(\mu,\la)$, 
it was proved in \cite{IvanovNewYork3517-3530}.
We recall this result below, see (\ref{f36}).

\subsection{Coherent systems and up/down Markov chains}\label{s1.3}
Here we give definitions of a coherent system on the Schur graph and 
of up/down Markov chains associated to it. We follow
\cite[\S1]{Borodin2007}.

The {\em{}down transition function\/}
for $\mu,\la\in\Sb$ such that $|\la|=|\mu|+1$
is 
\begin{equation}\label{f2}
	p^\downarrow(\la,\mu):=\frac{\dhh(\mu)}{\dhh(\la)}\kappa(\mu,\la).
\end{equation}
It can be easily checked that
\begin{itemize}
	\item $p^\downarrow(\la,\mu)\ge0$ for all $\mu,\la\in\Sb$ such that $|\la|=|\mu|+1$;
	\item $p^\downarrow(\la,\mu)$ vanishes unless $\mu\nearrow\la$;
	\item if $|\la|=n\ge1$, then $\sum_{\mu\colon|\mu|=n-1}p^\downarrow(\la,\mu)=1$.
\end{itemize}

\begin{df}\rm{}
	A coherent system on $\Sb$
	is a system of
	probability
	measures $M_n$ on $\Sb_n$, $n\in\Nn$, consistent with the down transition function:
	\begin{equation}\label{f3}
		M_n(\mu)=\sum_{\la\colon\la\searrow\mu}p^\downarrow(\la,\mu)M_{n+1}(\la)\qquad
		\mbox{for all $n\in\Nn$ and $\mu\in\Sb_n$}.
	\end{equation}
	Here by $M_n(\mu)$ we denote the measure of a singleton $\left\{ \mu \right\}$.
\end{df}

Fix a coherent system $\left\{ M_{n} \right\}$.
The {\em{}up transition function\/}
for $\la,\nu\in\Sb$ such that 
$|\la|=n$, $|\nu|=n+1$, $n\in\Nn$, and $M_n(\la)\ne0$ is 
\begin{equation*}
	p^\uparrow(\la,\nu):=\frac{M_{n+1}(\nu)}{M_{n}(\la)}p^\downarrow(\nu,\la).
\end{equation*}
The up transition function depends on the choice of a coherent system.
Moreover, $\left\{ M_n \right\}$ and $p^\uparrow$
are consistent in a sense similar to (\ref{f3}):
\begin{equation}\label{f4}
	M_{n+1}(\nu)=\sum_{\textstyle\genfrac{}{}{0pt}{}{\la\colon\la\nearrow\nu}{M_n(\la)\ne0}}
	p^\uparrow(\la,\nu)M_n(\la)\qquad\mbox{for all $n\in\Nn$ and $\nu\in\Sb_{n+1}$}.
\end{equation}
\begin{df}\label{p1.6}\rm{}
	A system of measures $\left\{ M_n \right\}$, where $M_n$ is a probability measure on $\Sb_n$, $n\in\Nn$,
	is called {\em{}nondegenerate\/}, if $M_n(\la)>0$ for all $n\in\Nn$ and $\la\in\Sb_n$.
\end{df}
Let $\left\{ M_n \right\}$ be a nondegenerate coherent system on $\Sb$.
For all $n\in\N$ we define a Markov chain $T_n$ on the set $\Sb_n$
with the following
transition matrix:
\begin{equation*}
	T_n(\la,\widetilde\la):=\sum_{\nu\colon|\nu|=n+1}p^\uparrow(\la,\nu)p^\downarrow(\nu,\widetilde\la),\qquad
	\la,\widetilde\la\in\Sb_n.
\end{equation*}
This is the composition of the up and down transition functions,
from $\Sb_n$ to $\Sb_{n+1}$ and then back to $\Sb_n$.
From (\ref{f3}) and (\ref{f4}) it follows that 
$M_n$ is a stationary distribution for $T_n$. It can be readily shown
that the matrix $M_n(\la)T_n(\la,\widetilde\la)$
is symmetric with respect to the substitution $\la\leftrightarrow\widetilde\la$.
This means that the chain $T_n$ is reversible with respect to $M_n$.

\subsection{Multiplicative measures}\label{s1.4}
In this subsection we recall some definitions and results from \cite{Borodin1997}
concerning multiplicative measures on the Schur graph.
\begin{df}\label{p1.7}\rm{}
	For $n\in\Nn$ the {\em{}Plancherel measure\/} on the set $\Sb_n$ is defined as
	\begin{equation*}
		\mathrm{Pl}_n(\la):=\frac{\dhh(\la)^2}{n!}2^{\ell(\la)-n},\qquad \la\in\Sb_n,
	\end{equation*}
	where $\dhh(\la)$ is given by (\ref{f1}).
\end{df}	
The Plancherel measures form a nondegenerate 
coherent system $\left\{ \mathrm{Pl}_n \right\}$ on $\Sb$.
\begin{df}\rm{}\label{p1.8}
	Let $M_n$ be a probability measure on $\Sb_n$ for all $n\in\Nn$.
	The system of measures $\left\{ M_n \right\}$ is called {\em{}multiplicative\/} if 
	\begin{equation}\label{f5}
		M_n(\la)=\mathrm{Pl}_n(\la)\cdot\frac1{Z(n)}\cdot\prod_{\square\in\la}
		f\left( \mi(\square),\mj(\square) \right)\quad
		\mbox{for all $n\in\Nn$ and $\la\in\Sb_n$}
	\end{equation}
	for some functions
	$f\colon\Nn^2\to\mathbb{C}$ and $Z\colon\Nn\to\mathbb{C}$.
	Here the product is taken
	over all boxes in the shifted diagram $\la$, and the numbers
	$\mi(\square)$ and $\mj(\square)$ are the 
	row and column numbers of the box $\square$, respectively.\footnote{The
	row number is counted from up to down, and the column number is counted from
	left to right.}
\end{df}
\begin{thm}[Borodin \cite{Borodin1997}]
	A nondegenerate 
	multiplicative system of probability measures $\left\{ M_n \right\}$
	is coherent if and only if 
	the functions $f(i,j)$ and $Z(n)$ from (\ref{f5}) have the form
	\begin{equation}\label{f6}
		\begin{array}{rcrcl}
			f(i,j)&=& f_\al(i,j)&:=& (j-i)(j-i+1)+\al,\\
			Z(n)&=& Z_\al(n)&:=& \al(\al+2)(\al+4)\dots(\al+2n-2)
		\end{array}
	\end{equation}
	for some parameter $\al\in(0,+\infty]$.
\end{thm}
We denote the multiplicative coherent system corresponding to
$\al$ by $\left\{ M_n^\al \right\}$
Below we call this object simply the {\em{}multiplicative measures\/}.\footnote{
Note that for all $\la\in\Sb_n$ 
the ratio $\prod_{\square\in\la}f_{\al}(\mi(\square),\mj(\square))/Z_{\al}(n)$ tends to one as $\al\to+\infty$.
Thus, one can say that $\left\{ M_n^\infty \right\}$ 
coincides with the Plancherel coherent system.}
The up transition function
corresponding to $\left\{ M_n^\al \right\}$
can be written out explicitly:
\begin{equation}\label{f7}
	p_\al^\uparrow(\la,\nu)=
	\frac{\mc(\nu/\la)\left( \mc(\nu/\la)+1 \right)+\al}{2|\la|+\al}\cdot
	\frac{\dhh(\nu)}{\dhh(\la)\left( |\la|+1 \right)}.
\end{equation}
where $\al\in(0,+\infty]$ and $\mc(\square):=\mj(\square)-\mi(\square)$ is the {\em{}content\/} of the box $\square$.

\begin{rmk}\rm{}
	One can consider {\em{}degenerate multiplicative measures\/}.
	That is, for certain negative values of $\al$
	the formulas (\ref{f5})--(\ref{f6}) define a system of measures $\left\{ M_n \right\}$
	not on the whole Schur graph $\Sb$, but on a certain finite subset of $\Sb$.

	Namely, if $\al=\al_N:=-N(N+1)$ for some $N\in\N$,
	then $M_n^{\al_N}$ is a probability measure on $\Sb_n$ for 
	all $n=0,1,\dots,\frac{N(N+1)}2$. The system $\left\{ M_n^{\al_N} \right\}$
	satisfies (\ref{f3}) for $n=0,1,\dots,\frac{N(N+1)}2$.
	It is clear from (\ref{f5})--(\ref{f6}) that $M_n^{\al_N}(\la)>0$ iff $\la_1\le N$.

	Thus, one can say that $\left\{ M_n^{\al_N} \right\}_{n=0,1,\dots,\frac{N(N+1)}2}$
	is a coherent system of measures on the finite graded graph 
	$\Sb(N):=\left\{ \la\in\Sb\colon\la_1\le N \right\}\subset \Sb$.
\end{rmk}
The existence of 
degenerate multiplicative coherent systems 
is a useful observation, but in the present paper we concentrate on the case $\al\in(0,+\infty)$.
\begin{df}\label{p10.6}\rm{}
	In the rest of the paper the parameter $\al$ takes values in 
	$(0,+\infty)$. From now on by $T_n$ we denote the one-step transition operator 
	of the $n$th up/down Markov chain corresponding to the multiplicative 
	measures
	with parameter $\al$. This operator $T_n$ acts on functions on $\Sb_n$ (see \S\ref{s2.3} below
	for more detail).
\end{df}

\section{Kerov interlacing coordinates\\ of shifted Young diagrams}\label{s20}
In this section we introduce Kerov interlacing coordinates of shifted
Young diagrams and study their basic properties. 
These coordinates are similar to interlacing coordinates of ordinary Young diagrams 
introduced by S.~Kerov,
see \cite{Kerov2000}. 
In \S\ref{s20.2} and \S\ref{s20.3} we express the Schur graph's Plancherel
up transition function $p_{\infty}^{\uparrow}$ and the down transition function
$p^{\downarrow}$, respectively, 
in terms of Kerov interlacing coordinates.
This approach is similar to that 
explained in \cite{Kerov2000} and used in \cite{Olshanski2009},
but there are some significant differences.

\subsection{Definition and basic properties}\label{s20.1}
Let $\la\in\Sb_n$, $n\ge1$.
Denote by $X(\la)$ the set of numbers $\left\{ \mc(\nu/\la)\colon\nu\searrow\la \right\}$,
that is, $X(\la)$ is the set of 
contents of all boxes that can be added to the shifted Young diagram $\la$.
For every $x\in X(\la)$ there exists a unique shifted 
diagram $\nu\searrow\la$ such that $\mc(\nu/\la)=x$.
Denote this diagram $\nu$ by $\la+\square(x)$.
Similarly, let $Y(\la):=\left\{ \mc(\la/\mu)\colon\mu\nearrow\la \right\}$ be the 
set of contents of all boxes that can be removed from the shifted Young diagram $\la$.
For every $y\in Y(\la)$
there exists a unique shifted diagram $\mu\nearrow\la$ such that $\mc(\la/\mu)=y$. Denote 
this diagram $\mu$ by~$\la-\square(y)$.

For $\la=\varnothing$ we set
$X(\varnothing):=\left\{ 0 \right\}$,
$Y(\varnothing):=\varnothing$.

\begin{df}\rm{}\label{p1.1}
	Let $\la\in\Sb$.
	Suppose that the sets $X(\la)$ and $Y(\la)$ are written in ascending order.
	The numbers $\left[ X(\la);Y(\la) \right]$ are called 
	{\em{}Kerov coordinates\/} of a shifted Young diagram $\la$.
\end{df}

Figure \ref{fig1a} shows Kerov coordinates of two different shifted Young diagrams.
Namely, for $\mu=(6,5,1)$ Kerov coordinates are $X(\mu)=\left\{ 1,6 \right\}$ and 
$Y(\mu)=\left\{ 0,4 \right\}$ (Figure \ref{fig1a}a-b); and for $\nu=(6,5,3)$ these coordinates 
are
$X(\nu)=\left\{ 0,3,6 \right\}$ and 
$Y(\nu)=\left\{ 2,4 \right\}$ (Figure \ref{fig1a}c-d).
\begin{figure}[htpb]
	\begin{center}
		\includegraphics{abcd.eps}
	\end{center}
	\caption{Figure \ref{fig1a}.}
	\label{fig1a}
\end{figure}
\begin{prop}[The interlacing property]\label{p1.2}
	Let $\la$ be a shifted Young diagram.

	{\bf{}(a)\/} If $\la$ contains a one-box row 
	(see, for example, Figure \ref{fig1a}a-b), then for some integer 
	$d\ge1$ we have
	\begin{equation*}
		X(\la)=\left\{ x_1,\dots,x_d \right\},\qquad Y(\la)=\left\{ 0,y_2,\dots,y_d \right\}
	\end{equation*}
	and
	\begin{equation*}
		0=y_1<x_1<y_2<x_2<\dots<y_d<x_d.
	\end{equation*}

	{\bf{}(b)\/} If $\la$ does not contain a one-box row
	(see, for example, Figure \ref{fig1a}c-d), then for some 
	integer $d\ge0$ we have\footnote{Note that $d=0$ only for $\la=\varnothing$.}
	\begin{equation*}
		X(\la)=\left\{ 0,x_1,\dots,x_d \right\},\qquad Y(\la)=\left\{ y_1,\dots,y_d \right\}
	\end{equation*}
	and
	\begin{equation*}
		0=x_0<y_1<x_1<y_2<x_2<\dots<y_d<x_d.
	\end{equation*}
\end{prop}
\begin{proof}
	This can be simply proved by induction on the number of boxes of $\la$,
	by consecutively adding a box to the diagram.
	During this procedure the change
	of the number of one-box rows in the diagram\footnote{This number is 
	always zero or one.}
	leads to transition from the case
	{\bf{}(a)\/} to the case {\bf{}(b)\/} and vice versa.	
\end{proof}
\begin{rmk}\label{p1.3}\rm{}
	In the case of ordinary Young diagrams (see, e.g., \cite{Kerov2000,Olshanski2009})
	the number of elements in the set $X(\la)$ is always greater by one than the number
	of elements in the set $Y(\la)$. 
	In our case it is not always true.
	
	Let us define $X'(\la):=X(\la)\setminus\left\{ 0 \right\}$. 
	It is clear that
	the numbers of elements in the sets
	$X'(\la)$ and $Y(\la)$ are equal for all shifted diagrams $\la$.
	We will use this fact below.
\end{rmk}
\begin{rmk}\label{p1.4}\rm{}
	As can be also proved by induction on the number of boxes (similarly to 
	Proposition \ref{p1.2}),
	a shifted Young diagram $\la$ is uniquely
	determined by its Kerov coordinates $\left[ X(\la);Y(\la) \right]$, or,
	equivalently, by the pair of sequences $X'(\la)$ and $Y(\la)$.
\end{rmk}
\begin{rmk}\label{p8.9}\rm{}
	Let $\la$ be a nonempty shifted Young diagram, and
	\begin{equation*}
		X'(\la)=\left\{ x_1,\dots,x_d \right\},\qquad Y(\la)=\left\{ y_1,\dots,y_d \right\}
	\end{equation*}
	for some integer $d\ge1$. 
	It can be easily seen that $\la$ has the form
	\begin{equation*}
		\la=(x_d,x_d-1,\dots,y_d+1,x_{d-1},x_{d-1}-1,\dots,y_{d-1}+1,\dots,x_1,x_1-1,\dots,y_1+1)
	\end{equation*}
	(see Figure \ref{fig2}).
	Here for all $j$ the numbers $x_j,x_j-1,\dots,y_j+1$ are
	consecutive
	decreasing integers 
	(for some $j$ it can happen that $x_j=y_j+1$).
	Note that $y_1$ can be zero, this corresponds
	to the case {\bf{}(a)\/} in Proposition~\ref{p1.2}. 
	\begin{figure}[htpb]
		\begin{center}
			\includegraphics{sum-sum.eps}
		\end{center}
		\caption{Figure \ref{fig2}.}
		\label{fig2}
	\end{figure}
\end{rmk}
\begin{prop}\label{p1.5}
	For every shifted Young diagram $\la$ we have
	\begin{equation*}
		\sum_{x\in X(\la)}x(x+1)-\sum_{y\in Y(\la)}y(y+1)=2|\la|.
	\end{equation*}
\end{prop}
\begin{proof}
	If $\la=\varnothing$, the claim is obvious.
	Suppose $\la\ne0$. We have
	\begin{equation*}
		\sum_{x\in X(\la)}x(x+1)-\sum_{y\in Y(\la)}y(y+1)=
		\sum_{j=1}^{d}\big(x_j(x_j+1)-y_j(y_j+1)\big),
	\end{equation*}
	where the notation is as in Remark \ref{p8.9}.
	For all~$j$ the value $x_j(x_j+1)-y_j(y_j+1)$ clearly equals twice the area
	of the part of the shifted Young diagram $\la$ formed by the rows $x_j,x_j-1,\dots,y_j+1$
	(see Figure~\ref{fig2}).
	This concludes the proof.	
\end{proof}

\subsection{The Plancherel up transition function}\label{s20.2}
The up transition function corresponding to the Plancherel 
coherent system on~$\Sb$ (see \S\ref{s1.4})
can be written in terms of Kerov interlacing coordinates
of shifted Young diagrams.

Let $\la$ be an arbitrary shifted Young diagram and $v$ be a complex variable. By definition, put
\begin{equation}\label{f9.9}
	\Rc^\uparrow(v;\la):=\frac{\prod_{y\in Y(\la)}(v-y(y+1))}{v\cdot\prod_{x\in X'(\la)}(v-x(x+1))}.
\end{equation}
It follows from Remark \ref{p1.3} that 
the degree of the denominator is always
greater than the degree of the numerator 
by one.
Next, from Proposition \ref{p1.2} it follows that
if $\la$ contains a one-box row, then the numerator and the denominator
of $\Rc^\uparrow(v;\la)$ can both be divided by the factor $v$, and if $\la$ does not contain
a one-box row, then the fraction in the RHS of (\ref{f9.9}) is irreducible.
In either case, the denominator of the irreducible form of the fraction
$\Rc^\uparrow(v;\la)$ 
is equal to $\prod_{x\in X(\la)}(v-x(x+1))$.

Let $\trp^\uparrow_x(\la)$, $x\in X(\la)$, be the following expansion coefficients
of 
$\Rc^\uparrow(v;\la)$ 
as a sum of partial fractions:
\begin{equation}\label{f10}
	\Rc^\uparrow(v;\la)=
	\sum_{x\in X(\la)}\frac{\trp_x^\uparrow(\la)}{v-x(x+1)}.
\end{equation}

\begin{prop}\label{p1.10}
	For every shifted Young diagram $\la$ and all $x\in X(\la)$ we have
	\begin{equation*}
		\trp^\uparrow_x(\la)=p_{\infty}^{\uparrow}(\la,\la+\square(x)),
	\end{equation*}
	where $p_\infty^\uparrow(\cdot,\cdot)$ is the up transition function
	corresponding 
	to the Plancherel coherent system on $\Sb$ (see \S\ref{s1.4}).
\end{prop}
\begin{proof}
	It follows from (\ref{f10}) by the residue formula that 
	\begin{equation}\label{f800}
		\trp_{\widehat x}^\uparrow(\la)=\left\{
		\begin{array}{ll}
			\displaystyle\frac{\prod_{y\in Y(\la)}
			\left( \widehat x(\widehat x+1)-y(y+1) \right)}
			{\widehat x(\widehat x+1)\prod_{x \in X'(\la),\ x\ne\widehat x}
			\left( \widehat x(\widehat x+1)-x(x+1) \right)},&
			\mbox{if $\widehat x\ne0$};
			\\\rule{0pt}{22pt}
			\displaystyle
			\frac{\prod_{y\in Y(\la)}y(y+1)}{\prod_{x\in X'(\la)}x(x+1)},&
			\mbox{if $\widehat x=0$}
		\end{array}
		\right.
	\end{equation}
	for every $\widehat x\in X(\la)$.

	Taking the limit as $\al\to+\infty$ in (\ref{f7}),
	we obtain the following expression for the Plancherel up transition function:
	\begin{equation*}
		p_{\infty}^{\uparrow}(\la,\la+\square(\widehat x))=\frac{\dhh(\la+\square(\widehat x))}{\dhh(\la)\left( |\la|+1 \right)},\qquad
		\widehat x\in X(\la),
	\end{equation*}
	where $\dhh$ is given by (\ref{f1}).
	
	Let us check that the two above expressions coincide. Assume first that $\widehat x\ne0$.
	Let $\la=(\la_1,\dots,\la_{\ell})$
	and $\la+\square(\widehat x)=(\la_1,\dots,\la_{k-1},\la_k+1,\la_{k+1},\dots,\la_{\ell})$
	for some $1\le k\le \ell$. Note that $\la_k=\widehat x$. 
	Using (\ref{f1}), we have
	\begin{equation*}
		\left.
		\begin{array}{l}
			\displaystyle
			p_{\infty}^{\uparrow}(\la,\la+\square(\widehat x))=\frac{\dhh(\la+\square(\widehat x))}{\dhh(\la)\left( |\la|+1 \right)}
			\\\displaystyle\qquad=\rule{0pt}{20pt}
			\frac{2^{|\la|+1-\ell(\la+\square(\widehat x))}(|\la|+1)!}{(\widehat x+1)!\cdot\prod_{i\ne k}\la_i!}
			\times\\\displaystyle\qquad\quad\times
			\prod_{i=1}^{k-1}\frac{\la_i-\widehat x-1}{\la_i+\widehat x+1}
			\cdot\prod_{j=k+1}^{\ell}\frac{\widehat x+1-\la_j}{\widehat x+1+\la_j}
			\cdot\prod_{\textstyle\genfrac{}{}{0pt}{}{1\le i<j\le \ell}{i,j\ne k}}
			\frac{\la_i-\la_j}{\la_i+\la_j}\times\\\displaystyle\qquad\quad\times
			\frac{\widehat x!\cdot\prod_{i\ne k}\la_i!}{2^{|\la|-\ell(\la)}(|\la|+1)\cdot|\la|!}
			\cdot\prod_{i=1}^{k-1}\frac{\la_i+\widehat x}{\la_i-\widehat x}
			\cdot\prod_{j=k+1}^{\ell}\frac{\widehat x+\la_j}{\widehat x-\la_j}
			\cdot\prod_{\textstyle\genfrac{}{}{0pt}{}{1\le i<j\le \ell}{i,j\ne k}}
			\frac{\la_i+\la_j}{\la_i-\la_j}\\\displaystyle\qquad=
			\frac{2^{\ell(\la)-\ell(\la+\square(\widehat x))+1}}{\widehat x+1}
			\cdot
			\prod_{\textstyle\genfrac{}{}{0pt}{}{1\le i\le \ell}{i\ne k}}
			\frac{\widehat x(\widehat x+1)-\la_i(\la_i-1)}{\widehat x(\widehat x+1)-\la_i(\la_i+1)}.
		\end{array}
		\right.
	\end{equation*}
	Using Remark \ref{p8.9}, one can decompose the last product as follows:
	\begin{equation*}
		\prod_{\textstyle\genfrac{}{}{0pt}{}{1\le i\le \ell}{i\ne k}}
		\frac{\widehat x(\widehat x+1)-\la_i(\la_i-1)}{\widehat x(\widehat x+1)-\la_i(\la_i+1)}=
		\prod_{m=1}^{d}
		\prod_{\textstyle\genfrac{}{}{0pt}{}{r=y_m+1}{r\ne \widehat x}}^{x_m}
		\frac{\widehat x(\widehat x+1)-r(r-1)}{\widehat x(\widehat x+1)-r(r+1)},
	\end{equation*}
	where $X'(\la)=\left\{ x_1,\dots,x_d \right\}$ and $Y(\la)=\left\{ y_1,\dots,y_d \right\}$.

	Fix $m=1,\dots,d$. It can be readily verified that
	\begin{equation*}
		\prod_{\textstyle\genfrac{}{}{0pt}{}{r=y_m+1}{r\ne \widehat x}}^{x_m}
		\frac{\widehat x(\widehat x+1)-r(r-1)}{\widehat x(\widehat x+1)-r(r+1)}=
		\left\{
		\begin{array}{ll}\displaystyle
			\frac{\widehat x(\widehat x+1)-y_m(y_m+1)}{\widehat x(\widehat x+1)-x_m(x_m+1)},&\widehat x\ne x_m;\\
			\rule{0pt}{22pt}
			\displaystyle
			\frac1{2\widehat x}\big(\widehat x(\widehat x+1)-y_m(y_m+1)\big),&\widehat x=x_m.
		\end{array}
		\right.
	\end{equation*}

	Observe that if $\widehat x\ne0$, then $\ell(\la)=\ell(\la+\square(\widehat x))$.
	It can be readily verified that the above expression (\ref{f800})
	for $\trp^\uparrow_{\widehat x}(\la)$ coincides with 
	\begin{equation*}
		p_{\infty}^{\uparrow}(\la,\la+\square(\widehat x))=
		\frac{2}{\widehat x+1}
		\prod_{m=1}^{d}
		\prod_{\textstyle\genfrac{}{}{0pt}{}{r=y_m+1}{r\ne \widehat x}}^{x_m}
		\frac{\widehat x(\widehat x+1)-r(r-1)}{\widehat x(\widehat x+1)-r(r+1)}.
	\end{equation*}

	The case $\widehat x=0$ can be considered similarly
	with the observation that in this case $\ell(\la+\square(\widehat x))=\ell(\la)+1$.
	This concludes the proof.
\end{proof}	
\begin{rmk}\rm{}\label{p1.11}
	If we take the limit transition as $v\to\infty$ in (\ref{f10}), we obtain
	\begin{equation*}
		\lim_{v\to\infty}v\Rc^\uparrow(v;\la)=1=\lim_{v\to\infty}
		\sum_{x\in X(\la)}\frac{v\cdot p_{\infty}^{\uparrow}(\la,\la+\square(x))}{v+x(x+1)}	
		=\sum_{x\in X(\la)}p_{\infty}^{\uparrow}(\la,\la+\square(x))
	\end{equation*}
	for all $\la\in\Sb$ and $x\in X(\la)$, as it should be.
\end{rmk}
Note that now (\ref{f7}) can be rewritten as (here $\al\in(0,+\infty]$)
\begin{equation}\label{f14}
	p_\al^\uparrow(\la,\la+\square(x))=\frac{x(x+1)+\al}{2|\la|+\al}\cdot\trp^\uparrow_x(\la)\qquad
	\mbox{for all $\la\in\Sb$ and $x\in X(\la)$}.
\end{equation}

\subsection{The down transition function}\label{s20.3}
The down transition function of the Schur graph
can be written in terms of Kerov interlacing coordinates of shifted Young diagrams.

Let $\la$ be an arbitrary nonempty shifted Young diagram and $v$ be a complex
variable. By definition, put
\begin{equation*}
	\Rc^\downarrow(v;\la):=\frac1{v\Rc^\uparrow(v;\la)}=
	\frac{\prod_{x\in X'(\la)}(v-x(x+1))}{\prod_{y\in Y(\la)}(v-y(y+1))}.
\end{equation*}
Observe that the numerator and the denominator
both have $v^d$ as the term of maximal degree in $v$, where $d\ge0$ is the number of elements
in the set $X'(\la)$ (or, equivalently, in $Y(\la)$, see Remark \ref{p1.3}).

Let $\trp^\downarrow_y(\la)$, $y\in Y(\la)$,
be the following expansion coefficients of
$\Rc^\downarrow(v;\la)$ 
as a sum of partial fractions:
\begin{equation}\label{f15}
	\Rc^\downarrow(v;\la)=1-\sum_{y\in Y(\la)}\frac{\trp_y^{\downarrow}(\la)}{v-y(y+1)}.
\end{equation}
\begin{prop}\label{p1.12}
	For every nonempty shifted Young diagram $\la$ we have
	\begin{equation*}
		\trp_{y}^\downarrow(\la)=2|\la|\cdot p^\downarrow(\la,\la-\square(y)),\qquad y\in Y(\la),
	\end{equation*}
	where $p^\downarrow(\cdot,\cdot)$ is the down transition
	function (see \S\ref{s1.3} for the definition).
\end{prop}
\begin{proof}
	It follows from (\ref{f15}) by the residue formula that
	\begin{equation*}
		\trp_{\widehat y}^\downarrow(\la)=
		\frac{\prod_{x\in X'(\la)}\left( \widehat y(\widehat y+1)-x(x+1) \right)}
		{\prod_{y\in Y(\la),\ y\ne\widehat y}\left( \widehat y(\widehat y+1)-y(y+1) \right)}\qquad
		\mbox{for every $\widehat y\in Y(\la)$}.
	\end{equation*}
	Next, we can rewrite the definition of the down transition function (\ref{f2}) as
	\begin{equation*}
		p^{\downarrow}(\la,\la-\square(\widehat y))=\frac{\dhh(\la-\square(\widehat y))}{\dhh(\la)}\cdot 2^{1-\delta(\widehat y)},
	\end{equation*}
	where $\delta(\cdot)$ is the Kronecker delta, and the function $h$ is given by (\ref{f1}).

	It can be shown exactly as in the proof of Proposition \ref{p1.10}
	(using Remark~\ref{p8.9})
	that the two above expressions coincide.
\end{proof}

\section{The up/down Markov chains\\ and doubly symmetric functions}\label{s2}
In this section we 
compute the action of the operators $T_n$ from Definition \ref{p10.6}
on doubly symmetric functions (Theorem \ref{p2.7}).
We argue similarly to \cite[\S4]{Borodin2007}.

\subsection{Doubly symmetric functions}\label{s2.1}
In this subsection we briefly recall the definitions
of the algebra of doubly symmetric functions and some related objects.
Exact definitions and proofs concerning this subject 
can be found, e.g., in the paper by V.~Ivanov \cite{IvanovNewYork3517-3530}. 
See also \cite[Ch. III, \S8]{Macdonald1995} and~\cite{Stembridge1985}.

Let $\Lambda$ denote the algebra 
of real symmetric functions
in (formal) variables $y_1,y_2,\dots$.
This algebra is freely generated (as a commutative unital algebra)
by Newton power sums
$p_k:=\sum_{i=1}^{\infty}y_i^k$, $k\in\N$.
We write $\Lambda=\mathbb{R}\left[ p_1,p_2,p_3,\dots \right]$.
By $\Gamma$ we denote the subalgebra of $\Lambda$ generated by the odd Newton power
sums, $\Gamma=\mathbb{R}\left[ p_1,p_3,p_5,\dots \right]$.
We call $\Gamma$ the {\em{}algebra of doubly symmetric functions\/}.
\begin{rmk}\rm{}\label{p2.1a}
  The subalgebra of $\Lambda$ 
  generated by the odd Newton power sums was studied by various authors.
  However, there is no common notation for it. For example, in 
  \cite{IvanovNewYork3517-3530} and 
  \cite{Macdonald1995}
  it is denoted by $\Gamma$, in \cite{Hoffman1992}~---~by~$\Delta$, 
  in the papers \cite{Stembridge1989,Stembridge1992}
  --- by $\Omega$, and in a recent paper \cite{Berele2009} --- by~$\mathcal{D}$.
  In \cite{IvanovNewYork3517-3530}
  it is called the algebra of supersymmetric functions, and in \cite{Berele2009}
  --- the algebra of doubly symmetric functions. 
  In the present paper we adopt the latter term and
  the notation $\Gamma$ for this algebra.

  We do not use the term ``supersymmetric functions'' because it was used
  by J.~Stembridge \cite{Stembridge1985} in a different sense.
  Namely, he   studied the unital algebra generated by 
  the following {\em{}supersymmetric power sums\/}\footnote{The
  definitions and related discussions 
  can also be found in \cite{Macdonald1995}.} in two sets of variables $u_i$ and $v_j$:
  \begin{equation*}
    p_k(u_1,u_2,\dots;v_1,v_2,\dots)=\sum_{i=1}^{\infty}u_i^k-\sum_{j=1}^{\infty}v_j^k,\qquad
    k=1,2,\dots.
  \end{equation*}
  The algebra $\Gamma$ defined above is generated by 
  supersymmetric power sums in variables 
  $\left\{ {y_1},{y_2},\dots \right\}$ and 
  $\left\{ -y_1,-y_2,\dots \right\}$.
  Clearly, $\Gamma$ can also be viewed as a subalgebra of that supersymmetric algebra.	

  The algebra 
  $\Gamma$ consists of all $f\in\Lambda$ such that
  for every $1\le i<j$ the expression 
  \begin{equation*}
    f(y_1,\dots,y_{i-1},z,y_{i+1},\dots,y_{j-1},-z,y_{j+1},\dots)
  \end{equation*}
  does not depend on $z$ (here $z$ is another independent formal variable).\footnote{It 
  is clear that the odd Newton power sums satisfy that property 
  and the even do not. The fact that every $f\in\Lambda$ satisfying
  that property is a polynomial in the odd Newton power sums 
  follows from \cite{Stembridge1985}.}
  From \cite{Berele2009} it follows that $\Gamma$ can be viewed as the quotient
  of $\Lambda$ by the ideal generated by all $s_\si-s_{\si'}$, where $s_\si$
  is the ordinary Schur function, $\si$ runs over all ordinary partitions
  and $\si'$ denotes the conjugate of the partition $\si$.
\end{rmk}

There is a natural filtration of the algebra $\Lambda$ by degrees of polynomials
in formal variables $y_i$. This filtration is
determined by setting $\deg p_k=k$, $k\in\N$.
The subalgebra $\Gamma\subset\Lambda$ 
inherits this filtration from $\Lambda$ and thus becomes a filtered algebra
with the filtration determined by setting $\deg p_{2m-1}=2m-1$, $m\in\N$.
More precisely, 
$$
\Gamma=\bigcup_{m=0}^{\infty}\Gamma^{(m)},\qquad\Gamma^{(0)}\subset\Gamma^{(1)}\subset\Gamma^{(2)}\subset\ldots\subset \Gamma,
$$
where $\Gamma^{(m)}$ is the finite-dimensional subspace of $\Gamma$ consisting of
elements of degree $\le m$:
\begin{equation*}
	\Gamma^{(0)}=\mathbb{R}1,\qquad \Gamma^{(m)}=\mathrm{span}\left\{ p_1^{r_1}p_3^{r_3}\dots\colon
	r_1+3r_3+\dots\le m\right\},\quad m=1,2,\dots.
\end{equation*}
Finite products of the form
$p_1^{r_1}p_3^{r_3}\dots$
constitute a linear 
basis for $\Gamma$ as a vector space over $\mathbb{R}$.
Every element 
$p_1^{r_1}p_3^{r_3}\dots$ is homogeneous.
We will need two more linear bases for $\Gamma$, 
of which one is also homogeneous and the other is not.

\begin{df}[Schur's $\mQ$-functions]\rm{}\label{p2.2}
	Let $\la=(\la_1,\la_2,\dots,\la_{\ell(\la)},0,0,\dots)$ be an arbitrary strict partition.
	For every $n\ge\ell(\la)$ set
	\begin{equation}\label{f20}
		R_{\la\mid n}(y_1,\dots,y_n):=y_1^{\la_1}\dots y_{\ell(\la)}^{\la_{\ell(\la)}}
		\cdot\prod_{\textstyle\genfrac{}{}{0pt}{}{i\le\ell(\la)}{i<j\le n}}
		\frac{y_i+y_j}{y_i-y_j}.
	\end{equation}
	If $n\ge\ell(\la)$, define\footnote{Here $\mathfrak{S}_n$ is the symmetric group.}
	\begin{equation}\label{f21}
		\mQ_\la(y_1,\dots,y_n,0,\dots):=
		\frac{2^{\ell(\la)}}{(n-\ell(\la))!}\sum_{w\in\mathfrak{S}_n}
		R_{\la\mid n}(y_{w(1)},\dots,y_{w(n)}),
	\end{equation}
	and 
	$\mQ_\la(y_1,\dots,y_n,0,\dots):=0$ otherwise.
	The expressions
	$\mQ_\la(y_1,\dots,y_n,0,\dots)$, $n\in\N$,
	define a doubly symmetric function 
	$\mQ_\la\in\Gamma$.\footnote{This follows from \cite{IvanovNewYork3517-3530}. Note 
	that that paper deals with
	Schur's $\mathcal{P}$-functions. They are linear multiples of the
	$\mQ$-functions: $\mathcal{P}_\la=2^{-\ell(\la)}\mQ_\la$,
	$\la\in\Sb$.} It is called {\em{}Schur's $\mQ$-function\/}.
\end{df}
Each $\mQ_\la$, $\la\in\Sb$, is a homogeneous element of degree $|\la|$. The system 
$\left\{ \mQ_\la \right\}_{\la\in\Sb}$ is a linear basis for the algebra $\Gamma$ over $\mathbb{R}$.
\begin{df}[Factorial Schur's $\mQ$-functions]\rm{}\label{p2.3}
	The factorial analogues  
	of Schur's $\mQ$-functions are defined as in (\ref{f20})--(\ref{f21}),
	with $R_{\la\mid n}$ replaced by
	\begin{equation*}
		R_{\la\mid n}^*(y_1,\dots,y_n):=y_1^{\downarrow\la_1}\dots y_{\ell(\la)}^{\downarrow\la_{\ell(\la)}}
		\cdot\prod_{\textstyle\genfrac{}{}{0pt}{}{i\le\ell(\la)}{i<j\le n}}
		\frac{y_i+y_j}{y_i-y_j}.
	\end{equation*}
	Here $y_i^{\downarrow \la_i}$
	is the decreasing factorial power defined as
	$a^{\downarrow k}:=a(a-1)\dots(a-k+1)$, $k\in\N$, $a^{\downarrow 0}:=1$.
	The functions $\mQ^*_\la$, $\la\in\Sb$, are called {\em{}factorial Schur's $\mQ$-functions\/}.
\end{df}
For all $\la\in\Sb$ we have
$\mQ_\la^*=\mQ_\la+g$,
where $g$ is a doubly symmetric function with $\deg g<|\la|=\deg\mQ_\la$.
It follows that the system
$\left\{ \mQ^*_\la \right\}_{\la\in\Sb}$ is also a linear
basis for $\Gamma$ as a vector space over $\mathbb{R}$.

\subsection{A representation of $\mathfrak{sl}(2,\mathbb{C})$}\label{s2.2}
By $\mathrm{Fun}_0(\Sb)$ denote the algebra of real finitely supported functions
on the Schur graph $\Sb$ with pointwise operations.
A natural basis for $\mathrm{Fun}_0(\Sb)$ is $\left\{ \basf_\mu \right\}_{\mu\in\Sb}$,
where
\begin{equation*}
	\basf_\mu(\la)=\left\{
	\begin{array}{ll}
		1,&\mbox{if $\la=\mu$};\\
		0,&\mbox{otherwise}.
	\end{array}
	\right.
\end{equation*}
Let $E$, $F$, and $H$ be the following operators in $\mathrm{Fun}_0(\Sb)$
which are similar to Kerov's operators (see \cite{Okounkov2001a} 
for definition):\footnote{Here $\delta(\cdot)$
is the Kronecker delta.}
\begin{equation*}
	\left.
	\begin{array}{rcl}
		E\basf_\la&:=& \displaystyle\sum_{x\in X(\la)}2^{-\delta(x)}
		\left( x(x+1)+\al \right)\basf_{\la+\square(x)};\\
		F\basf_\la&:=& \displaystyle
		-\sum_{y\in Y(\la)}\basf_{\la-\square(y)};\\
		H\basf_\la&:=& \displaystyle
		\left( \frac\al2+2|\la| \right)\basf_\la.
	\end{array}
	\right.
\end{equation*}
\begin{lemma}\label{p2.4}
	For all $\al\in\mathbb{R}$
	these operators satisfy the commutation relations
	\begin{equation}\label{f30}
		\left[ E,H \right]=-2E,\qquad
		\left[ F,H \right]=2F,\qquad
		\left[ E,F \right]=H.
	\end{equation}
\end{lemma}
\begin{proof}
	The proof uses the results of 
	\S\ref{s20.1} and
	is similar to the proof of Lemma 4.2 of the paper \cite{Borodin2007}.
\end{proof}
\begin{corollary}\rm{}\label{p2.5}
	The correspondence
	\begin{equation*}
		\left(
		\begin{array}{cc}
			0&1\\0&0
		\end{array}
		\right)\to E,\qquad
		\left(
		\begin{array}{cc}
			0&0\\1&0
		\end{array}
		\right)\to F,\qquad
		\left(
		\begin{array}{cc}
			1&0\\0&-1
		\end{array}
		\right)\to H
	\end{equation*}
	defines a representation of the Lie algebra $\mathfrak{sl}(2,\mathbb{C})$
	in the space $\mathrm{Fun}_0(\Sb)$.
\end{corollary}

\begin{lemma}\label{p2.6}
	Fix $N\in\N$.
	Let
	$V_N$
	be the finite-dimensional subspace of $\mathrm{Fun}_0(\Sb)$
	spanned by the basis vectors $\basf_\la$ with $\la_1\le N$.\footnote{In fact, 
	$\dim V_N=N(N+1)/2$.}
	If $\al=-N(N+1)$, then $V_N$ is
	invariant under the action of the operators $E$, $F$, and $H$, and 
	the action of $\mathfrak{sl}(2,\mathbb{C})$ in $V_N$ defined in Corollary \ref{p2.5}
	lifts to a representation of the group $SL(2,\mathbb{C})$ in $V_N$.
\end{lemma}
\begin{proof}
	This can be proved exactly as Lemma 4.3 of the paper \cite{Borodin2007}.
\end{proof}

\subsection{The action of $T_n$ on factorial Schur's $\mQ$-functions}\label{s2.3}
For every set $\mathfrak{X}$ by $\mathrm{Fun}(\mathfrak{X})$
denote the algebra of real-valued 
functions on $\mathfrak{X}$
with pointwise operations.

Consider an embedding of the algebra $\Gamma$ described in \S\ref{s2.1}
into the algebra $\mathrm{Fun}(\Sb)$.
This embedding is defined on the generators 
of $\Gamma$:
\begin{equation*}
	p_k\to p_k(\la):=\sum_{i=1}^{\ell(\la)}\la_i^k,\qquad k=1,3,5,\dots.
\end{equation*}
Thus, to every element $f\in\Gamma$ corresponds a function from $\mathrm{Fun}(\Sb)$.
Denote this function by $f(\la)$.
We identify the (abstract) algebra $\Gamma$ 
with its image under this embedding, that is, with
the algebra of functions
$\left\{ f(\cdot)\in\mathrm{Fun}(\Sb)\colon f\in\Gamma \right\}$.
\begin{rmk}\label{90.9}\rm{}
	In \cite{Borodin2007} the role of $\Gamma$ is 
	played by the algebra generated by supersymmetric
	power sums (see Remark \ref{p2.1a}) in
	$a_i$ and $-b_j$, where $a_i$ and $b_j$ are the modified Frobenius coordinates
	of an ordinary Young diagram.
	The paper \cite{Olshanski2009} deals
	with Jack deformations of these power sums.
\end{rmk}
For any $f\in\Gamma$, by $f_n$ denote the restriction of the function $f(\cdot)$
to $\Sb_n\subset\Sb$. It can be easily checked that the algebra 
$\Gamma\subset\mathrm{Fun}(\Sb)$
separates points of $\Sb$.
It follows that the functions of the form $f_n$, with $f\in\Gamma$, 
exhaust the (finite-dimensional) space $\mathrm{Fun}(\Sb_n)$,
$n\in\Nn$.

Our aim in this section is to prove the following
\begin{thm}\label{p2.7}
	Let $T_n\colon\mathrm{Fun}(\Sb_n)\to\mathrm{Fun}(\Sb_n)$, $n\in\N$,
	be the operator from Definition \ref{p10.6}.
	Its action on the functions $(\mQ_\mu^*)_n$, $\mu\in\Sb$,
	is as follows:
	\begin{equation}\label{f32}
		\left.
		\begin{array}{l}
			\displaystyle
			(T_n-{\bf1})(\mQ_\mu^*)_n=\frac1{(n+1)(n+\al/2)}\Bigg[
			-|\mu|\left( |\mu|+\al/2-1 \right)(\mQ_\mu^*)_n\\\qquad\qquad\qquad\qquad\ \displaystyle+
			(n-|\mu|+1)\sum_{y\in Y(\mu)}\left( y(y+1)+\al \right)
			(\mQ_{\mu-\square(y)}^*)_n
			\Bigg],
		\end{array}
		\right.
	\end{equation}
	where ${\bf1}$ denotes the identity operator.
\end{thm}
\begin{rmk}\rm{}\label{p2.8}
	The above Theorem states that $(T_n-{\bf1})(\mQ_\mu^*)_n$ (for all $\mu\in\Sb$)
	is a linear combination of the function $(\mQ_\mu^*)_n$
	and the functions of the form $(\mQ_\varkappa^*)_n$,
	where $\varkappa$ runs over all shifted diagrams 
	that can be obtained from $\mu$ by deleting one box.
	Recall (\S\ref{s20.1}) that these diagrams are indexed by the set $Y(\mu)$.
\end{rmk}

The proof of Theorem \ref{p2.7} uses the technique 
from \cite[\S4]{Borodin2007}. We do not want to repeat all details
and in the rest of the section we give a scheme of the proof.

Fix arbitrary $n\in\N$ and $\al\in(0,+\infty)$. 
We write $T_n$ as the composition of ``down'' 
$D_{n+1,n}\colon\mathrm{Fun}(\Sb_{n})\to\mathrm{Fun}(\Sb_{n+1})$
and ``up''
$U_{n,n+1}\colon\mathrm{Fun}(\Sb_{n+1})\to\mathrm{Fun}(\Sb_{n})$
operators acting on functions:
\begin{equation}\label{f33}
	\left.
	\begin{array}{rcll}
		\left( D_{n+1,n}f_n \right)(\la)&:=& 
		\displaystyle
		\sum_{\mu\colon\mu\nearrow\la}
		p^\downarrow(\la,\mu)f_n(\mu),&\la\in\Sb_{n+1};\\
		\left( U_{n,n+1}f_{n+1} \right)(\nu)&:=&
		\displaystyle
		\sum_{\varkappa\colon\varkappa\searrow\nu}
		p^\uparrow_\al(\nu,\varkappa)f_{n+1}(\varkappa),&\nu\in\Sb_n.
	\end{array}
	\right.
\end{equation}
The operator $D_{n+1,n}$ is constructed using the down
transition function $p^\downarrow$ and does not depend 
on the parameter $\al$.
The operator $U_{n,n+1}$ is constructed using the 
up transition function $p^\uparrow_\al$
and therefore depends on the parameter $\al$.

\begin{rmk}\rm{}
	These ``down'' and ``up'' operators act on functions.
	They are adjoint to the corresponding operators
	acting on measures. The latter act in accordance
	with their names, for example, the operator
	$D_{n+1,n}^*$
	maps $\mathcal{M}(\Sb_{n+1})$ into $\mathcal{M}(\Sb_n)$,
	where 
	$\mathcal{M}(\mathfrak{X})$ denotes the space of 
	measures on $\mathfrak{X}$.
\end{rmk}

It clearly follows from the definition of 
the $n$th up/down Markov chain (\S\ref{s1.3}) that 
$T_n=U_{n,n+1}\circ D_{n+1,n}\colon\mathrm{Fun}(\Sb_n)\to\mathrm{Fun}(\Sb_n)$.
We deal with the operators $D_{n+1,n}$ and $U_{n,n+1}$ separately.

\begin{lemma}[The operator $D$]\label{p2.9}
	There exists a unique operator
	$D\colon\Gamma\to\Gamma$ such that
	\begin{equation*}
		D_{n+1,n}f_n=\frac1{n+1}\left( Df \right)_{n+1}
	\end{equation*}
	for all $n\in\Nn$ and $f\in\Gamma$. In the basis
	$\{\mQ_\mu^*\}_{\mu\in\Sb}$ for the algebra $\Gamma$ this operator has the form
	\begin{equation}\label{f34}
		D\mQ_\mu^*=(p_1-|\mu|)\mQ_\mu^*.
	\end{equation}
\end{lemma}
\begin{proof}
	The proof is exactly the same as the proof of Theorem 4.1 (1) of the paper \cite{Borodin2007},
	but instead of the facts about Frobenius-Schur functions we refer to the following formula 
	which is due to V.~Ivanov \cite{IvanovNewYork3517-3530}.
	
	Let $|\la|=n$, $\mu\in\Sb$ and $|\mu|\le n$. Then
	\begin{equation}\label{f36}
		\frac{\dhh(\mu,\la)}{\dhh(\la)}=2^{-|\mu|}
		\frac{(\mQ_\mu^*)_n(\la)}{n(n-1)\dots(n-|\mu|+1)}.
	\end{equation}
	We also use the recurrence relations for 
	the function
	$\dhh(\mu,\la)$ which directly
	follow from its definition (\S\ref{s1.1}):
	\begin{equation*}
		\dhh(\mu,\nu)=\sum_{\la\colon\la\nearrow\nu}
		\dhh(\mu,\la)\kappa(\la,\nu)
		\qquad\mbox{for all $\mu,\nu\in\Sb$}.
	\end{equation*}
	The rest of the proof repeats that of \cite[Theorem 4.1 (1)]{Borodin2007}.
\end{proof}
\begin{lemma}[The operator $U$]\label{p2.10}
	For every $\al\in(0,+\infty)$ there exists a unique operator 
	$U\colon\Gamma\to\Gamma$ depending on $\al$ such that
	\begin{equation*}
		U_{n,n+1}f_{n+1}=\frac{1}{n+\al/2}(Uf)_n
	\end{equation*}
	for all $n\in\Nn$ and $f\in\Gamma$.
	In the basis
	$\{\mQ_\mu^*\}_{\mu\in\Sb}$ for the algebra $\Gamma$ this operator has the form
	\begin{equation}\label{f37}
		\left.
		\begin{array}{l}
			\displaystyle
			U\mQ_\mu^*=
			\left( p_1+|\mu|+\frac\al2 \right)\mQ_\mu^*+
			\sum_{y\in Y(\mu)}\big(y(y+1)+\al\big)\mQ_{\mu-\square(y)}^*.
		\end{array}
		\right.
	\end{equation}
\end{lemma}
\begin{proof}
	The proof is similar to that of Theorem 4.1 (2) of the paper \cite{Borodin2007}.

	We must prove that
	\begin{equation}\label{f40}
		\left.
		\begin{array}{l}
			\displaystyle
			\left( n+\frac\al2 \right)
			(U_{n,n+1}(\mQ_\mu^*)_{n+1})(\la)=
			\left( n+k+\frac\al2 \right)(\mQ_\mu^*)_n(\la)
			\\\displaystyle\qquad\qquad\qquad
			+
			\sum_{y\in Y(\mu)}
			\big(y(y+1)+\al\big)(\mQ_{\mu-\square(y)}^*)_n(\la)
		\end{array}
		\right.
	\end{equation}
	for all $\mu,\la\in\Sb$ such that $|\mu|=k$
	and $|\la|=n\ge k$.

	If $|\mu|=0$, that is, $\mu$ is an empty partition,
	then $\mQ_\mu^*\equiv1$ and (\ref{f40}) clearly holds.

	Now let $|\mu|=k\ge1$. Using (\ref{f7}), (\ref{f36}) and the definition of
	$U_{n,n+1}$ one can
	reduce (\ref{f40}) to the following equivalent combinatorial
	identity:
	\begin{equation*}
		\begin{array}{l}
			\displaystyle
			\sum_{x\in X(\la)}\big(x(x+1)+\al\big)\dhh(\mu,\la+\square(x))\\\displaystyle\qquad
			=
			2\left( n+k+\frac\al2 \right)\left( n-k+1 \right)\dhh(\mu,\la)\\\displaystyle\qquad\qquad+
			\sum_{y\in Y(\mu)}\big(y(y+1)+\al\big)\dhh(\mu-\square(y),\la),
		\end{array}
	\end{equation*}
	where $|\la|=n$ and $|\mu|=k\le n$.

	This combinatorial identity is verified exactly as the corresponding
	identity from the proof of \cite[Theorem 4.1 (2)]{Borodin2007}. 
	In our case one must use the results formulated in \S\ref{s2.2}.
\end{proof}
Theorem \ref{p2.7} now follows from Lemmas \ref{p2.9} and \ref{p2.10} 
and the fact that 
\begin{equation}\label{f41}
	(T_n-{\bf1})f_n=-f_n+
	\frac{(UDf)_n}{(n+1)(n+\al/2)},\qquad f\in\Gamma.
\end{equation}

\section{Doubly symmetric functions\\ on shifted Young diagrams}\label{s3}
In this section we study the algebra $\Gamma\subset\mathrm{Fun}(\Sb)$ 
(defined in \S\ref{s2}) in more detail.

Let $\la$ be an arbitrary shifted Young diagram and $u$ be a complex variable.
By definition, put
\begin{equation*}
	\phi(u;\la):=\prod_{i=1}^{\infty}\frac{u+\la_i}{u-\la_i}.
\end{equation*}
Note that this product is actually finite, because any strict partition $\la$
has only finitely many nonzero terms.
Note also that 
$\phi(u;\la)$ is a rational function in $u$  
taking value $1$ at $u=\infty$.
\begin{prop}\label{p3.1}
	The algebra $\Gamma\subset\mathrm{Fun}(\Sb)$ 
	coincides with the commutative 
	unital subalgebra of $\mathrm{Fun}(\Sb)$
	generated by the Taylor expansion coefficients of $\phi(u;\la)$
	(or, equivalently, of $\log\phi(u;\la)$) at $u=\infty$ with respect to $u^{-1}$.
\end{prop}
\begin{proof}
	The Taylor expansion of $\log\phi(u;\la)$ at $u=\infty$ has the form
	\begin{equation}\label{f44}
		\log\phi(u;\la)=2\sum_{k\ge1\ \mbox{odd}}\frac{p_k(\la)}{k}u^{-k},
	\end{equation}
	where $p_k(\la)=\sum_{i=1}^{\ell(\la)}\la_i^{k}$ are the Newton power sums.
	The algebra $\Gamma$ is freely generated by the functions $p_1,p_3,\ldots\in\mathrm{Fun}(\Sb)$,
	see \S\ref{s2}.
\end{proof}

By definition, put\footnote{Here $v$ is an independent
complex variable.} 
\begin{equation*}
	\Phi(v;\la):=\prod_{i=1}^{\infty}\frac{v-\la_i(\la_i-1)}{v-\la_i(\la_i+1)}.
\end{equation*}
The product here is also actually finite.
Clearly, $\Phi(v;\la)$
is a rational function in $v$ taking value $1$ at $v=\infty$.
It can be readily verified that 
\begin{equation*}
	\Phi(u^2-u;\la)=\frac{\phi(u-1;\la)}{\phi(u;\la)}.
\end{equation*}
\begin{df}\rm{}\label{p3.2}
	Let $\p_m(\cdot), \g_m(\cdot), \gh_m(\cdot)\in\mathrm{Fun}(\Sb)$, $m\in\N$,
	be the following Taylor expansion coefficients at $v=\infty$ with respect to $v^{-1}$:
	\begin{equation*}
		\left.
		\begin{array}{rcl}
			\log\Phi(v;\la)&=&\displaystyle
			\sum_{m=1}^{\infty}\frac{\p_m(\la)}{m}v^{-m};\\
			\Phi(v;\la)&=& \displaystyle 1+\sum_{m=1}^{\infty}\g_m(\la)v^{-m};\\
			\displaystyle\frac1{\Phi(v;\la)}&=& \displaystyle
			1-\sum_{m=1}^{\infty}\gh_m(\la)v^{-m}.
		\end{array}
		\right.
	\end{equation*}
\end{df}
Recall that the algebra $\Gamma$ has a natural filtration (defined in \S\ref{s2.1}) which is 
determined by setting 
\begin{equation}\label{f45}
	\deg p_{2m-1}=2m-1,\qquad m=1,2,\dots.
\end{equation}
\begin{prop}\label{p3.3}
	The functions $\p_m(\la)$ belong to the algebra $\Gamma$.
	More precisely, 
	\begin{equation*}
		\p_m(\la)=2m\cdot p_{2m-1}(\la)+\dots,\qquad m\in\N,
	\end{equation*}
	where dots stand for lower degree terms 
	in the algebra $\Gamma$, which are a
	linear combination of 
	$p_{2l-1}(\la)$, where $1\le l\le m-1$.
\end{prop}
\begin{proof}
	On one hand, 
	by the definition of $\Phi$ and by (\ref{f44}) we have
	\begin{equation*}
		\left.
		\begin{array}{rcl}\displaystyle
			\log\Phi(u^2-u;\la)&=& \log\phi(u-1;\la)-\log\phi(u;\la)\\&=&\displaystyle
			2\sum_{k=1}^{\infty}\frac{p_{2k-1}(\la)}{2k-1}
			\left( \frac1{(u-1)^{2k-1}}-\frac1{u^{2k-1}} \right)
		\end{array}
		\right.
	\end{equation*}
	for all $\la\in\Sb$.
	Observe that
	\begin{equation*}
		\frac1{(u-1)^{2k-1}}-\frac1{u^{2k-1}}=
		(2k-1)u^{-2k}\left( 1+\frac{k}{u}+\dots \right),
	\end{equation*}
	where $k\in\N$ and dots stand for terms containing $u^{-2},u^{-3},\dots$.

	On the other hand, by Definition \ref{p3.2} we have
	\begin{equation*}
		\log\Phi(u^2-u;\la)=\sum_{m=1}^{\infty}\frac{\p_m(\la)}{m}\frac1{(u^2-u)^m}
	\end{equation*}
	for all $\la\in\Sb$. Observe that
	\begin{equation*}
		\frac1{(u^2-u)^{m}}=u^{-2m}\left( 1-\frac mu+\dots \right),
	\end{equation*}
	where $m\in\N$ and again dots stand for 
	terms containing $u^{-2},u^{-3},\dots$.

	Thus, we get the following identity:
	\begin{equation*}
		\begin{array}{l}\displaystyle
			2\sum_{k=1}^{\infty}u^{-2k}p_{2k-1}(\la)\left( 1+\frac ku+\dots \right)
			=\sum_{m=1}^{\infty}u^{-2m}\frac{\p_m(\la)}{m}\left( 1-\frac{m}{u}+\dots \right).
		\end{array}
	\end{equation*}
	Comparing the coefficients of $u^{-2m}$ in both sides, we get the claim.
\end{proof}
\begin{prop}\label{p3.4}
	We have\footnote{Here and below	
	we sometimes omit the argument $\la$ to shorten the notation.}
	\begin{equation*}
		\g_1=\gh_1=\p_1
	\end{equation*}
	and 
	\begin{equation*}
		k\g_k=\p_k+\p_{k-1}\g_1+\dots+\p_1\g_{k-1},\qquad
		\gh_k=\g_k-\g_{k-1}\gh_1-\dots-\g_1\gh_{k-1}
	\end{equation*}
	for all $k=2,3,\dots$.
\end{prop}
\begin{proof}
	The technique of this proof is similar to \cite[Ch. I, \S2]{Macdonald1995}.
	Let $w$ be an independent variable.
	Observe that
	\begin{equation*}
		\sum_{m=1}^{\infty}\frac{\p_m(\la)}mw^m=\log\left( 1+\sum_{k=1}^{\infty}\g_k(\la)w^k \right).
	\end{equation*}
	If we take $d/dw$ of both sides
	and compare the coefficients by $w^{k-1}$, we get the desired relation between
	$\p_k$'s and $\g_k$'s.

	To prove the remaining 
	relation between $\g_k$'s and $\gh_k$'s observe that
	\begin{equation*}
		\left( 1+\sum_{k=1}^{\infty}\g_k(\la)w^k \right)
		\left( 1-\sum_{k=1}^{\infty}\gh_k(\la)w^k \right)=1.
	\end{equation*}
	This concludes the proof.
\end{proof}
\begin{corollary}\label{p3.5}
	Each of the three families
	$\left\{ \p_1,\p_2,\p_3,\dots \right\}$,
	$\left\{ \g_1,\g_2,\g_3,\dots \right\}$ and
	$\left\{ \gh_1,\gh_2,\gh_3,\dots \right\}$
	is a system of algebraically independent 
	generators of the algebra $\Gamma$. Under the identification
	of $\Gamma$ with any of the algebras of polynomials
	\begin{equation*}
		\mathbb{R}\left[ \p_1,\p_2,\dots \right],\quad
		\mathbb{R}\left[ \g_1,\g_2,\dots \right]\quad\mbox{and}\quad
		\mathbb{R}\left[ \gh_1,\gh_2,\dots \right],
	\end{equation*}
	the natural filtration (\ref{f45}) of $\Gamma$ is determined by setting
	\begin{equation*}
		\left.
		\begin{array}{rcl}
			\deg\p_m(\la)&=& 2m-1,\\
			\deg\g_m(\la)&=& 2m-1,\\
			\deg\gh_m(\la)&=& 2m-1,\qquad m\in\N,
		\end{array}
		\right.
	\end{equation*}
	respectively.
\end{corollary}
\begin{prop}\label{p3.6}
	Let $\la$ be
	an arbitrary shifted Young diagram with Kerov interlacing coordinates
	$\left[ X(\la);X(\la) \right]$ (see \S\ref{s20.1}).
	Then
	\begin{equation*}
		\Phi(v;\la)=
		\frac{\prod_{y\in Y(\la)}(v-y(y+1))}{\prod_{x\in X'(\la)}(v-x(x+1))}=
		v\cdot\Rc^\uparrow(v;\la).
	\end{equation*}
	Here the function $\Rc^\uparrow$ is defined by (\ref{f9.9}).
	Recall that $X'(\la)=X(\la)\setminus\left\{ 0 \right\}$.
\end{prop}
\begin{proof}
	This can be proved exactly as Proposition \ref{p1.10} using Remark \ref{p8.9}. 
\end{proof}
Using this Proposition one can express
the functions $\p_m,\g_m,\gh_m$, $m\in\N$, through Kerov coordinates:
\begin{prop}\label{p3.7}
	Let $\la\in\Sb$ and $m\in\N$. Then\footnote{Recall that the numbers 
	$\{ \trp_x^\uparrow(\la) \}_{x\in X(\la)}$ and 
	$\{ \trp_y^\downarrow(\la) \}_{y\in Y(\la)}$ were introduced in \S\ref{s20}.}
	\begin{equation*}
		\left.
		\begin{array}{rcl}
			\displaystyle
			\p_m(\la)&=& \displaystyle
			\sum_{x\in X(\la)}\left( x(x+1) \right)^{m}-
			\sum_{y\in Y(\la)}\left( y(y+1) \right)^{m};\\
			\g_m(\la)&=& \displaystyle
			\sum_{x\in X(\la)}\trp_x^\uparrow(\la)\cdot\left( x(x+1) \right)^{m};\\
			\gh_m(\la)&=& \displaystyle
			\sum_{y\in Y(\la)}\trp_y^\downarrow(\la)\cdot
			\left( y(y+1) \right)^{m-1}.
		\end{array}
		\right.
	\end{equation*}
\end{prop}
\begin{proof}
	The first claim is a straightforward consequence of Proposition \ref{p3.6}.

	Let us prove the second claim.
	On one hand, from the definition of the numbers $\{\trp_x^{\uparrow}\}$
	(see \S\ref{s20.2}) we have
	\begin{equation*}
		\left.
		\begin{array}{l}\displaystyle
			v\cdot\Rc^\uparrow(v;\la)=v\sum_{x\in X(\la)}\frac{\trp_x^\uparrow(\la)}{v-x(x+1)}
			=\sum_{x\in X(\la)}\frac{\trp_x^\uparrow(\la)}{1-\frac{x(x+1)}v}\\
			\displaystyle\qquad=
			\sum_{x\in X(\la)}\trp_x^\uparrow(\la)\sum_{k=0}^{\infty}
			\left( \frac{x(x+1)}v \right)^{k}
			\\\displaystyle\qquad=
			\sum_{k=0}^{\infty}v^{-k}\sum_{x\in X(\la)}\trp_x^\uparrow(\la)\cdot\left( x(x+1) \right)^{k}.
		\end{array}
		\right.
	\end{equation*}
	On the other hand, it follows from 
	Proposition \ref{p3.6} that
	$v\cdot\Rc^\uparrow(v;\la)=\Phi(v;\la)$.
	Using the definition of the functions $\g_m$ (Definition \ref{p3.2})
	and comparing it to the above formula for $v\cdot\Rc^\uparrow(v;\la)$,
	we get the second claim.

	The third claim can be verified similarly.
\end{proof}
It follows from Propositions \ref{p1.5}, \ref{p3.4} and \ref{p3.7} that 
\begin{equation}\label{f54}
	\p_1(\la)=\g_1(\la)=\gh_1(\la)=2|\la|,\qquad \la\in\Sb.
\end{equation}
\begin{lemma}\label{p3.8}
	Let $\la$ be an arbitrary 
	nonempty shifted Young diagram,
	$x\in X(\la)$ and $y\in Y(\la)$. 
	Then
	\begin{equation*}
		\begin{array}{rcl}\displaystyle
			\frac{\Phi(v;\la+\square(x))}{\Phi(v;\la)}&=& \displaystyle
			\frac{(v-x(x+1))^{2}}{\left( v-x(x+1) \right)^{2}-2(v+x(x+1))}
		\end{array}
	\end{equation*}
	and
	\begin{equation*}
		\begin{array}{rcl}\displaystyle
			\frac{\Phi(v;\la-\square(y))}{\Phi(v;\la)}&=& \displaystyle
			\frac{\left( v-y(y+1) \right)^{2}-2(v+y(y+1))}{(v-y(y+1))^{2}}.
		\end{array}
	\end{equation*}
\end{lemma}
\begin{proof}
	This directly follows from Proposition \ref{p3.6} 
	and the definitions of the diagrams $\la+\square(x)$ and $\la-\square(y)$ (\S\ref{s20.1}).
\end{proof}

\section{The up and down operators\\ in differential form}\label{s4}
The aim of this section is to write the operators
$D$ and $U$ in the algebra $\Gamma$
(they were defined in Lemmas
\ref{p2.9} and \ref{p2.10}, respectively)
in differential form. 
Here we use the results of \S\ref{s20} and \S\ref{s3}.
Our approach is inspired by the paper \cite{Olshanski2009}, but in our situation significant 
modifications are required.

\subsection{Formulation of the theorem}\label{s4.1}
We identify $\Gamma$ with the polynomial algebra
$\mathbb{R}\left[ \g_1,\g_2,\dots \right]$. 
Recall that $\Gamma$ is a filtered 
algebra, and under this identification the filtration
is determined by setting (see Corollary \ref{p3.5})
$\deg\g_m=2m-1$, $m\in\N$.
\begin{df}\label{p4.1}\rm{}
	We say that an operator $R\colon\Gamma\to\Gamma$ has degree $\le r$, 
	where $r\in\mathbb{Z}$, if
	$\deg(Rf)\le\deg f+r$ for any $f\in\Gamma$.
\end{df}
\begin{rmk}\rm{}\label{p40.2}
	Observe that any
	operator in the algebra of polynomials 
	(in finitely or countably many variables)
	can be written as a differential 
	operator with polynomial coefficients
	--- a formal infinite sum of differential monomials.
	This fact is well known and can be readily proved.
	We do not need it but it is useful to keep it in mind while 
	reading the formulation and the proof of Theorem~\ref{p4.2}.
\end{rmk}
\begin{thm}\label{p4.2}
	{\rm{}(1)\/}
	The operator $D\colon\Gamma\to\Gamma$ defined in Lemma \ref{p2.9} has degree $1$ with respect
	to the filtration of $\Gamma$ and looks as
	\begin{equation*}
		\left.
		\begin{array}{l}
			\displaystyle
			D=\frac12\g_1+
			\sum_{r,s\ge1}(2r-1)(2s-1)\g_{r+s-1}\frac{\partial^2}{\partial\g_r\partial\g_s}
			\\\displaystyle\qquad-
			\sum_{r\ge1}(2r-1)\g_r\frac{\partial}{\partial\g_r}+
			\sum_{r,s\ge1}(r+s)\g_r\g_s\frac{\partial}{\partial\g_{r+s}}\\\displaystyle\qquad{}+{}
			\mbox{\rm{}operators of degree $\le-2$};
		\end{array}
		\right.
	\end{equation*}

	{\rm{}(2)\/}
	For any fixed $\al\in(0,+\infty)$
	the operator $U\colon\Gamma\to\Gamma$ defined in Lemma \ref{p2.10} has degree $1$ with respect
	to the filtration of $\Gamma$ and looks as
	\begin{equation*}
		\left.
		\begin{array}{l}
			\displaystyle
			U=\frac12\g_1+\frac12\al+\al\frac{\partial}{\partial\g_1}+
			\sum_{r,s\ge1}(2r-1)(2s-1)\g_{r+s-1}\frac{\partial^2}{\partial\g_r\partial\g_s}
			\\\displaystyle\qquad+
			\sum_{r\ge1}(2r-1)\g_r\frac{\partial}{\partial\g_r}+
			\sum_{r,s\ge1}(r+s-1)\g_r\g_s\frac{\partial}{\partial\g_{r+s}}\\\displaystyle\qquad{}+{}
			\mbox{\rm{}operators of degree $\le-2$}.
		\end{array}
		\right.
	\end{equation*}
\end{thm}
\par\noindent{\em{}Scheme of proof.\/}
The functions $\g_k\in\Gamma$, $k\in\N$, generate the algebra $\Gamma$.
However, we will 
not be dealing with actions of $D$ and $U$ on these generators.
Instead, we consider the products of the form
$\Phi(v_1;\la)\Phi(v_2;\la)\dots$, 
where $v_1,v_2,\dots$
are independent complex variables in a finite number
(here $\Phi(v;\la)$ is defined in \S\ref{s3}).
It follows from Definition \ref{p3.2}
that the products of the form
$\Phi(v_1;\la)\Phi(v_2;\la)\dots$
assemble 
various products of the generators, which in turn
constitute a linear basis for $\Gamma$. 
Thus, we know the action of our operators if we
know how they transform such products.
It turns out that the transformation
of the products 
$\Phi(v_1;\la)\Phi(v_2;\la)\dots$
can be written down in a closed form. From
this we can extract all the necessary information.

\subsection{Action of $D$ and $U$ on generating series}\label{s4.2}
It follows from Definition \ref{p3.2}
that for any finite collection of independent complex variables $v_1,v_2,\dots$
(we prefer not to indicate their number explicitly) we have
\begin{equation}\label{f55}
	\Phi(v_1;\la)\Phi(v_2;\la)\ldots=\sum_{\rho}m_\rho(v_1^{-1},v_2^{-1},\dots)\g_\rho(\la),
\end{equation}
where the sum is taken over all ordinary partitions $\rho$ such that $\ell(\rho)$
does not exceed the number of $v_i$'s,
$m_\rho$ is the monomial symmetric function
and $\g_\rho=\g_{\rho_1}\dots\g_{\rho_{\ell(\rho)}}$.
It is convenient to set $\g_0:=1$.
\begin{rmk}\rm{}\label{p4.3}
	Observe that
	$m_\rho(v_1^{-1},v_2^{-1},\dots)$ vanishes if
	$\ell(\rho)$ is greater than the number of $v_i$'s.
	It follows that here and below
	in the sums
	similar to (\ref{f55})
	we can 
	let $\rho$ run over all ordinary partitions.
\end{rmk}
We regard the LHS of (\ref{f55}) as a generating series for the 
elements $\g_\rho$ that constitute a linear basis for $\Gamma$.
Thus, the action of an operator in $\Gamma$ on the LHS
is determined by its action on the functions $\g_\rho\in\Gamma$ in the
RHS.

Occasionally, it will be convenient to omit the
argument $\la$ in $\Phi(v;\la)$. 
Recall from \S\ref{s2.3}
the notation $(\dots)_n$ for
the restriction of a function from $\Gamma\subset\mathrm{Fun}(\Sb)$
to the subset $\Sb_n\subset\Sb$.

We start with the operators $U_{n,n+1}$ and $D_{n+1,n}$ defined by (\ref{f33}).
By the very definition of 
$U_{n,n+1}$ we have
\begin{equation*}
	\begin{array}{l}
		\displaystyle
		\left(
			U_{n,n+1}
			\Big(
				\prod_l\Phi(v_l)
			\Big)_{n+1}
		\right)(\la)
		\\\displaystyle\qquad
		=
		\sum_{x\in X(\la)}^{\phantom{A}}p_\al^\uparrow(\la,\la+\square(x))
		\prod_l\Phi(v_l;\la+\square(x)),
		\qquad\la\in\Sb_n.
	\end{array}
\end{equation*}
Using (\ref{f14}) and Lemma \ref{p3.8}, we get (note that $|\la|=n$):
\begin{equation}\label{f56}
	\begin{array}{l}
		\displaystyle
		\left(
			(n+\al/2)U_{n,n+1}
			\Big(
				\prod_l\Phi(v_l)
			\Big)_{n+1}\right)(\la)
		=F^\uparrow(v_1,v_2,\dots;\la)\cdot
		\prod_l\Phi(v_l;\la),
	\end{array}
\end{equation}
where
\begin{equation}\label{f57}
	\begin{array}{l}\displaystyle
		F^\uparrow(v_1,v_2,\dots;\la)
		\\\displaystyle\qquad:=\sum_{x\in X(\la)}\frac{x(x+1)+\al}2
		\prod_l\frac{(v_l-x(x+1))^2}
		{(v_l-x(x+1))^2-2(v_l+x(x+1))}
		\trp_x^\uparrow(\la).
	\end{array}
\end{equation}

Likewise, for the operator $D_{n+1,n}$ we have
\begin{equation*}
	\begin{array}{l}
		\displaystyle
		\left(
			D_{n+1,n}
			\Big(
				\prod_l\Phi(v_l)
			\Big)_{n}
		\right)(\la)
		\\\displaystyle\qquad=
		\sum_{y\in Y(\la)}^{\phantom{A}}
		p^\downarrow(\la,\la-\square(y))
		\prod_l\Phi(v_l;\la-\square(y)),
		\qquad\la\in\Sb_{n+1}.
	\end{array}
\end{equation*}
Using Proposition \ref{p1.12} and Lemma \ref{p3.8}, we get (note that now $|\la|=n+1$):
\begin{equation}\label{f58}
\begin{array}{l}
	\displaystyle
	\left(
		(n+1)D_{n+1,n}
		\Big(
			\prod_l\Phi(v_l)
		\Big)_{n}
	\right)(\la)
	=F^\downarrow(v_1,v_2,\dots;\la)\cdot
	\prod_l\Phi(v_l;\la),
\end{array}
\end{equation}
where
\begin{equation}\label{f59}
	\left.
	\begin{array}{l}
		\displaystyle
		F^\downarrow(v_1,v_2,\dots;\la):=\sum_{y\in Y(\la)}\frac12
		\prod_l\frac{(v_l-y(y+1))^2-2(v_l+y(y+1))}
		{(v_l-y(y+1))^2}
		\trp_y^\downarrow(\la).
	\end{array}
	\right.
\end{equation}
\begin{lemma}\label{p4.4}
	As functions in $\la$, both $F^\uparrow(v_1,v_2,\dots;\la)$
	and $F^\downarrow(v_1,v_2,\dots;\la)$ are elements of
	the algebra $\Gamma$. 
	More precisely, the both expressions can be viewed as 
	elements of $\Gamma\big[ [v_1^{-1},v_2^{-1},\dots] \big]$.
\end{lemma}
\begin{proof}
	Observe that the products on $l$ in (\ref{f57}) and (\ref{f59})
	can be viewed as elements of
	$\mathbb R[x(x+1)]\big[ [v_1^{-1},v_2^{-1},\dots] \big]$
	and 
	$\mathbb R[y(y+1)]\big[ [v_1^{-1},v_2^{-1},\dots] \big]$,
	respectively.\footnote{Here and below $\mathbb{R}\left[ z(z+1) \right]$ denotes the 
	algebra of polynomials in $z(z+1)$.}
	Moreover, if 
	$f\left( x(x+1) \right)$
	is a polynomial in $x(x+1)$,
	then the expression
	$$
		\sum_{x\in X(\la)}
		f\left( x(x+1) \right)
		\trp_x^\uparrow(\la)
	$$
	as function in $\la$ belongs to $\Gamma$ (this follows from Proposition \ref{p3.7}).
	Letting $f$ be the corresponding 
	formal power series in $v_1^{-1}, v_2^{-1},\dots$, we
	get the claim about $F^\uparrow(v_1,v_2,\dots;\la)$.
	The remaining claim about $F^\downarrow(v_1,v_2,\dots;\la)$ is verified similarly.
\end{proof}
Now we proceed to the operators $D$ and $U$ in the algebra $\Gamma$
defined in Lemmas \ref{p2.9} and \ref{p2.10}, respectively.
Using these Lemmas, we rewrite (\ref{f56}) and (\ref{f58}) as
\begin{equation}\label{f60}
	\begin{array}{rcl}
		U(\Phi(v_1)\Phi(v_2)\dots)&=& F^\uparrow(v_1,v_2,\dots)\Phi(v_1)\Phi(v_2)\dots;\\
		D(\Phi(v_1)\Phi(v_2)\dots)&=& F^\downarrow(v_1,v_2,\dots)\Phi(v_1)\Phi(v_2)\dots.
	\end{array}
\end{equation}
These formulas contain in a compressed form
all the information about the action
of $U$ and $D$ on the basis elements
$\g_\rho$, where $\rho$ runs over all ordinary partitions.
Our next step is to extract from 
(\ref{f60}) some explicit expressions
for $U\g_\rho$ and $D\g_\rho$
using (\ref{f55}) and Proposition \ref{p3.7}.

\subsection{Action of $U$ and $D$ in the basis $\left\{ \g_\rho \right\}$}\label{s4.3}
Let us first introduce some extra notation.
Let $v$ and $\xi$ be independent variables.
Consider the following expansions at $v=\infty$ with respect to $v^{-1}$:
\begin{equation}\label{f62}
	\begin{array}{lr}
		\displaystyle
		\frac{(v-\xi)^2}{(v-\xi)^2-2(v+\xi)}=
		\sum_{s=0}^\infty a_s(\xi)v^{-s},&
		a_s\in\mathbb R[\xi];
		\\\displaystyle\qquad\qquad
		\frac{(v-\xi)^2-2(v+\xi)}{(v-\xi)^2}=
		\sum_{s=0}^\infty b_s(\xi)v^{-s},&
		b_s\in\mathbb R[\xi].
	\end{array}
\end{equation}
\begin{lemma}\label{p4.5}
We have
	\begin{equation*}
		a_0(\xi)=b_0(\xi)\equiv1,
	\end{equation*}
	and $a_s(\xi)$ and $b_s(\xi)$ have degree $s-1$ for $s\ge1$. 
	More precisely,
	\begin{equation*}
		b_s(\xi)=-2(2s-1)\xi^{s-1},\qquad s\ge1,
	\end{equation*}
	and $a_s(\xi)$ has the form
	\begin{equation*}
		a_s(\xi)=2(2s-1)\xi^{s-1}+\mbox{\rm{}terms of degree $\le (s-2)$ in the variable $\xi$},\qquad s\ge1.
	\end{equation*}
\end{lemma}
\begin{proof}
	First, we compute explicitly $b_s(\xi)$:
	\begin{equation*}
		\begin{array}{l}
			\displaystyle
			\frac{(v-\xi)^2-2(v+\xi)}{(v-\xi)^2}=
			1-2\frac{v+\xi}{(v-\xi)^2}\\
			\displaystyle\qquad
			=1-\frac{4\xi}{(v-\xi)^2}-\frac2{v-\xi}=
			1-4\sum_{s=1}^\infty s\xi^s v^{-s-1}-
			2\sum_{s=0}^\infty \xi^s v^{-s-1}\\
			\displaystyle\qquad
			=1-\frac2v-\sum_{s=1}^\infty(4s+2)\xi^s v^{-s-1}=
			1-\sum_{s=1}^\infty2(2s-1)\xi^{s-1}v^{-s}.
		\end{array}
	\end{equation*}
	Next, observe that 
	$(\sum_{s=0}^\infty a_s(\xi)v^{-s})
	(\sum_{s=0}^\infty b_s(\xi)v^{-s})=1$,
	therefore, $a_0(\xi)=1$ and for $s\ge1$ the top degree term
	of $a_s(\xi)$ is equal to $2(2s-1)\xi^{s-1}$.
\end{proof}
For an ordinary partition
$\si=(\si_1,\si_2,\dots,\si_{\ell(\si)})$ we set
\begin{equation*}
	a_\si(\xi):=\prod_{i=1}^{\ell(\si)}a_{\si_i}(\xi),\qquad
	b_\si(\xi):=\prod_{i=1}^{\ell(\si)}b_{\si_i}(\xi).
\end{equation*}

Using (\ref{f62}) and the above definition, we get
the following expressions for the products on $l$ in (\ref{f57}) and (\ref{f59}):
\begin{equation}\label{f63}
	\begin{array}{lcr}
		\displaystyle
		\prod_l\frac{(v_l-x(x+1))^2}{(v_l-x(x+1))^2-2(v_l+x(x+1))}=
		\sum_\si a_\si(x(x+1))m_\si(v_1^{-1},v_2^{-1},\dots);
		\\\displaystyle
		\prod_l\frac{(v_l-y(y+1))^2-2(v_l-y(y+1))}{(v_l-y(y+1))^2)}=
		\sum_\si b_\si(y(y+1))m_\si(v_1^{-1},v_2^{-1},\dots).
	\end{array}
\end{equation}
Here the sums in the right-hand sides are taken over all 
ordinary partitions.

Next, introduce two linear maps
\begin{equation*}
	\left.
	\begin{array}{ll}
		\mathbb{R}\left[ x(x+1) \right]\to\Gamma,& f\mapsto\langle f\rangle^\uparrow;\\
		\mathbb{R}\left[ y(y+1) \right]\to\Gamma,& h\mapsto\langle h\rangle^\downarrow
	\end{array}
	\right.
\end{equation*}
by setting
\begin{equation}\label{f64}
	\big\langle(x(x+1))^m\big\rangle^\uparrow:=\g_m,
	\quad
	\big\langle(y(y+1))^m\big\rangle^\downarrow:=\gh_{m+1},
	\qquad
	m\in\Nn,
\end{equation}
where, by agreement, $\g_0=1$.
This definition is inspired by Proposition \ref{p3.7}.

Finally, let $c_{\si\tau}^\rho$
be the structure constants of the algebra $\Lambda$ 
of all symmetric functions
in the basis of monomial symmetric functions:
\begin{equation*}
	m_\si m_\tau=\sum_\rho c^\rho_{\si\tau}m_\rho.
\end{equation*}
Note that $c^\rho_{\si\tau}$ can be nonzero only if $|\rho|=|\si|+|\tau|$.
Here $\rho$, $\sigma$, and $\tau$ are ordinary partitions.

Now we are in a position to compute $U\g_\rho$ and $D\g_\rho$.
\begin{lemma}\label{p4.6}
	With the notation introduced above we have
	\begin{equation*}
		\begin{array}{rcl}
			\displaystyle
			U\g_\rho&=&\displaystyle
			\sum_{\si,\tau\colon
			|\si|+|\tau|=|\rho|}\frac12
			c^\rho_{\si\tau}\Big\langle\big(x(x+1)+\al\big)\cdot
			a_\si\big(x(x+1)\big)\Big\rangle^\uparrow
			\g_\tau;
			\\\displaystyle
			D\g_\rho&=& \displaystyle\sum_{\si,\tau\colon
			|\si|+|\tau|=|\rho|}\frac12
			c^\rho_{\si\tau}
			\Big\langle b_\si\big(y(y+1)\big)\Big\rangle^\downarrow\g_\tau.
		\end{array}
	\end{equation*}
\end{lemma}
\begin{proof}
	Let us write 
	\begin{equation*}
		\begin{array}{rcll}
			\displaystyle
			F^\uparrow(v_1,v_2,\dots)&=& \displaystyle\sum_\sigma F_\si^\uparrow \cdot
			m_\si(v_1^{-1},v_2^{-1},\dots),&\quad F_\si^\uparrow\in\Gamma;
			\\
			\displaystyle
			F^\downarrow(v_1,v_2,\dots)&=& \displaystyle\sum_\si F_\si^\downarrow \cdot
			m_\si(v_1^{-1},v_2^{-1},\dots),&\quad F_\si^\downarrow\in\Gamma,
		\end{array}
	\end{equation*}
	where sums are taken over all ordinary partitions $\si$.

	Using (\ref{f55}), we get 
	\begin{equation*}
		\begin{array}{l}
			\displaystyle
			\sum_\rho m_\rho(v_1^{-1},v_2^{-1},\dots)
			U\g_\rho\\
			\displaystyle\qquad
			=
			\bigg( \sum_\si 
			F_\si^\uparrow m_\si(v_1^{-1},v_2^{-1},\dots)\bigg)
			\bigg(\sum_\tau m_\tau(v_1^{-1},v_2^{-1},\dots)\g_\tau\bigg),
		\end{array}
	\end{equation*}
	which implies
	\begin{equation*}
		U\g_\rho=
		\sum_{\si,\tau\colon|\si|+|\tau|=|\rho|}c^\rho_{\si\tau}F_\si^\uparrow\g_\tau.
	\end{equation*}
	Similarly we obtain
	\begin{equation*}
		D\g_\rho=
		\sum_{\si,\tau\colon|\si|+|\tau|=|\rho|}c^{\rho}_{\si\tau}F_\si^\downarrow\g_\tau.
	\end{equation*}
	The facts that 
	\begin{equation*}
		F_\si^\uparrow=\Big\langle\frac12\big(x(x+1)+\al\big)\cdot
		a_\si\big(x(x+1)\big)\Big\rangle^\uparrow;\qquad
		F_\si^\downarrow=\Big\langle
		\frac12b_\si\big(y(y+1)\big)\Big\rangle^\downarrow
	\end{equation*}
	follow directly from (\ref{f57}), (\ref{f59}) and (\ref{f63}).
\end{proof}

\subsection{The operator $D$ in differential form}\label{s4.4}
In this subsection we prove claim {\rm{(1)}\/} of Theorem
\ref{p4.2}, that is, compute the top degree terms of the operator $D\colon\Gamma\to\Gamma$.

By virtue of Lemma \ref{p4.6}, we can write
\begin{equation}\label{f65}
	D=\sum_\si D_\si,\qquad
	D_\si\g_\rho:=\sum_{\tau\colon|\tau|=|\rho|-|\si|}
	\frac12
	\Big\langle b_\si\big(y(y+1)\big)\Big\rangle^\downarrow
	c^\rho_{\si\tau}\g_\tau.
\end{equation}
\begin{lemma}\label{p4.7}
	Let $\si$ be a nonempty ordinary partition. 
	Then
	\begin{equation*}
		\deg D_\si\le
		\max_{\rho,\tau}(\ell(\rho)-\ell(\tau)-2\ell(\si)+1),
	\end{equation*}
	where the maximum is taken over all pairs $(\rho,\tau)$
	such that $c^\rho_{\si\tau}\ne0$.

	A more rough but simpler estimate is
	\begin{equation*}
		\deg D_\si\le-\ell(\si)+1.
	\end{equation*}
\end{lemma}
\begin{proof}
	We have
	\begin{equation*}
		\begin{array}{l}
			\displaystyle
			\deg D_\si\le
			\max_{\rho,\tau}\big(\deg
			\left\langle b_\si\big(y(y+1)\big)\right\rangle^\downarrow
			+\deg\g_\tau-\deg\g_\rho\big)\\
			\displaystyle\qquad
			=\max_{\rho,\tau}\big(\deg
			\left\langle b_\si\big(y(y+1)\big)\right\rangle^\downarrow
			+2|\tau|-\ell(\tau)-2|\rho|+\ell(\rho)\big)\\
			\displaystyle\qquad
			\le
			\max_{\rho,\tau}\big(\deg\left\langle 
			b_\si\big(y(y+1)\big)\right\rangle^\downarrow
			-2|\si|-\ell(\tau)+\ell(\rho)\big),
		\end{array}
	\end{equation*}
	where maximums are taken over all pairs of ordinary partitions 
	$(\rho,\tau)$ such that $c^\rho_{\si\tau}\ne0$.
	The first line follows from the very definition
	of $\deg D_\si$, the second line holds 
	because $\deg\g_m=2m-1$, $m=1,2,\dots$
	(Corollary \ref{p3.5}), and therefore
	\begin{equation*}
		\deg\g_\rho=2|\rho|-\ell(\rho),\qquad \deg\g_\tau=2|\tau|-\ell(\tau)
	\end{equation*}
	for any ordinary partitions $\rho$ and $\tau$, and
	the
	third line holds because $c^{\rho}_{\si\tau}\ne0$ 
	only if $|\rho|=|\si|+|\tau|$.

	By assumption, $\si$ is nonempty, therefore, $\ell(\si)\ge1$. 
	Set $\si=(\si_1,\dots,\si_{\ell(\si)})$, where $\si_i\ge1$ for $i=1,\dots,\ell(\si)$,
	and write $\left\langle b_\si\big(y(y+1)\big)\right\rangle^\downarrow$ in more detail:
	\begin{equation*}
		\begin{array}{l}
			\displaystyle
			\left\langle b_\si\big(y(y+1)\big)\right\rangle^\downarrow
			\\\displaystyle\qquad=
			\left\langle \prod_{i=1}^{\ell(\si)} 
			b_{\si_i}\big(y(y+1)\big)\right\rangle^\downarrow
			=\left\langle \prod_{i=1}^{\ell(\si)}
			-2(2\si_i-1)\big(y(y+1)\big)^{\si_i-1}\right\rangle^\downarrow.
		\end{array}
	\end{equation*}
	We see that the polynomial $b_\si\big(y(y+1)\big)$ has degree $|\si|-\ell(\si)$
	in $y(y+1)$, therefore $\left\langle b_\si\big(y(y+1)\big)\right\rangle^\downarrow$
	is equal, within a nonzero scalar factor, to $\gh_{|\si|-\ell(\si)+1}$
	(Proposition \ref{p3.7}).
	Note that $\deg\gh_{|\si|-\ell(\si)+1}=2|\si|-2\ell(\si)+1$,
	therefore,
	\begin{equation*}
		\deg D_\si\le\max_{\rho,\tau}\big(2|\si|-2\ell(\si)+1-2|\si|-\ell(\tau)+\ell(\rho)\big),
	\end{equation*}
	which is the first estimate. To prove the second one, 
	observe that $c^\rho_{\si\tau}\ne0$ implies $\ell(\rho)\le\ell(\si)+\ell(\tau)$.
\end{proof}
From the second estimate of the above lemma follows a
\begin{corollary}\label{p4.8}
	If $\ell(\si)\ge3$, then $\deg D_\si\le-2$.
\end{corollary}
By this corollary, it suffices to examine 
the operators $D_\si$ with $\ell(\si)=0$ (that is, $\si=\varnothing$),
$\ell(\si)=2$, and $\ell(\si)=1$.
The next three lemmas 
consider these cases consequently.
\begin{lemma}\label{p4.9}
	$D_\varnothing=\displaystyle\frac12\g_1$.
\end{lemma}
\begin{proof}
	Suppose that $\si=\varnothing$ in (\ref{f65}). 
	Then
	$\tau=\rho$ and $c^\rho_{\si\tau}=1$.
	Moreover, since $b_\si=1$, then
	$\left\langle b_\si\big(y(y+1)\big)\right\rangle^\downarrow=\langle1\rangle^\downarrow=\gh_1$.
	By virtue of Proposition \ref{p3.4}, $\gh_1=\g_1$. 
	This concludes the proof.
\end{proof}
\begin{lemma}\label{p4.10}
	\begin{equation*}
		\sum_{\si\colon\ell(\si)=2}D_\si=
		\sum_{r,s\ge1}(2r-1)(2s-1)\g_{r+s-1}
		\frac{\partial^2}{\partial\g_r\partial\g_s}{}+{}
		\mbox{\rm{}operators of degree $\le-2$}.
	\end{equation*}
\end{lemma}
\begin{proof}
	Suppose that in (\ref{f65}) we have
	$\ell(\si)=2$, that is,
	$\si=(\si_1,\si_2)$, $\si_1\ge\si_2\ge1$.
	Lemma \ref{p4.7} shows that it must be 
	$\ell(\rho)=\ell(\tau)+2$, otherwise 
	the corresponding distribution to $D_\si$ has degree $\le-2$.
	This means that 
	\begin{equation*}
		\si_1=\rho_i,\quad \si_2=\rho_j,\qquad \tau=
		\left( \rho_1,\dots,\rho_{i-1},\rho_{i+1},\dots,\rho_{j-1},\rho_{j+1},\dots,\rho_{\ell(\rho)} \right)
	\end{equation*}
	for some $1\le i<j\le\ell(\rho)$.
	Using Lemma \ref{p4.5}, we get
	\begin{equation*}
		\big\langle b_\si\left( y\big(y+1\big) \right)\big\rangle^\downarrow=
		4(2\si_1-1)(2\si_{2}-1)\gh_{\si_1+\si_2-1}.
	\end{equation*}
	It follows that
	\begin{equation*}
		\left( \sum_{\si\colon\ell(\si)=2}D_\si \right)
		\g_\rho=
		\sum_{1\le i<j\le\ell(\rho)}
		2(2\rho_i-1)(2\rho_j-1)
		c^{\rho}_{\si\tau}
		\gh_{\rho_i+\rho_j-1}
		\g_{\rho\setminus\left\{ \rho_i,\rho_j \right\}}.
	\end{equation*}
	Note that for such $\rho$, $\si$, and $\tau$ as described above we have
	\begin{equation}\label{f68}
		c^{\rho}_{\si\tau}\g_{\rho\setminus\left\{ \rho_i,\rho_j \right\}}
		=\left\{
		\begin{array}{ll}
			\displaystyle
			\frac{\partial^2\g_{\rho}}{\partial\g_{\rho_i}\partial\g_{\rho_j}},
			&\text{if $\rho_i\ne\rho_j$};\\
			\rule{0pt}{24pt}
			\displaystyle
			\frac12\frac{\partial^2\g_{\rho}}{\partial\g_{\rho_i}^2},&
			\text{if $\rho_i=\rho_j$}.
		\end{array}
		\right.
	\end{equation}
	It follows that we can write
	\begin{equation*}
		\begin{array}{l}
			\displaystyle
			\sum_{\si\colon\ell(\si)=2}D_\si=
			\sum_{r_1>r_2\ge1}2(2r_1-1)(2r_2-1)\gh_{r_1+r_2-1}
			\frac{\partial^2}{\partial\g_{r_1}\g_{r_2}}
			\\
			\displaystyle\qquad
			+\frac12\sum_{r\ge1}2(2r-1)^2\gh_{2r-1}\frac{\partial^2}{\partial\g_r^2}.
		\end{array}
	\end{equation*}
	Using Proposition \ref{p3.4}, we can substitute each of $\gh_k$'s
	above by $\g_k$.
	Indeed, since
	$\gh_k=\g_k+\mbox{terms of degree $\le 2k-2$}$,
	$k=1,2,\dots$,
	then 
	replacing $\gh_k$ by $\g_k$ affects 
	only negligible terms (that is, summands of degree $\le-2$
	in the expression for the operator $D$).
	The result of this substitution is the desired expression.
\end{proof}
\begin{lemma}\label{p4.11}
	\begin{equation*}
		\begin{array}{l}
			\displaystyle
			\sum_{\si\colon\ell(\si)=1}D_\si=
			-\sum_{r\ge1}(2r-1)\g_r\frac{\partial}{\partial\g_r}
			\\\displaystyle\qquad
			+\sum_{r,s\ge1}(r+s)\g_r\g_s\frac\partial{\partial\g_{r+s}}
			{}+{}
			\mbox{\rm{}operators of degree $\le-2$}.
		\end{array}
	\end{equation*}
\end{lemma}
\begin{proof}
	Suppose that in (\ref{f65}) we have $\ell(\si)=1$, that is, $\si=(s)$ 
	for some $s\in\N$.

	It follows from Lemma \ref{p4.7} 
	that either $\ell(\rho)=\ell(\tau)+1$, or $\ell(\rho)=\ell(\tau)$.
	Let us examine these cases separately.

	Assume first that $\ell(\rho)=\ell(\tau)+1$. 
	This means that
	$$
		s=\rho_i,\qquad\tau=(\rho_1,\dots,\rho_{i-1},\rho_{i+1},\dots,\rho_{\ell(\rho)})
	$$
	for some $1\le i\le \ell(\rho)$.
	Using Lemma \ref{p4.5}, we get
	\begin{equation*}
		\left\langle
		b_\si\big(y(y+1)\big)\right\rangle^\downarrow
		=-2(2s-1)\gh_s.
	\end{equation*}
	Therefore, the case $\ell(\rho)=\ell(\tau)+1$ gives rise to the terms
	\begin{equation*}
		-\sum_{i=1}^{\ell(\rho)}(2\rho_i-1)c^{\rho}_{\si\tau}
		\gh_{\rho_i} \g_{\rho\setminus\left\{ \rho_i \right\}}.
	\end{equation*}
	Note that in this case we have
	$c^{\rho}_{\si\tau}\g_{\rho\setminus\left\{ \rho_i \right\}}=
	\partial\g_\rho/\partial\g_{\rho_i}$, and it follows that we obtain terms
	\begin{equation*}
		-\sum_{r\ge1}(2r-1)\gh_r\frac{\partial}{\partial\g_r}
	\end{equation*}
	contributing to $\sum_{\si\colon\ell(\si)=1}D_\si$.

	Now assume that $\ell(\rho)=\ell(\tau)$. This means that $\tau$
	is obtained from $\rho$ by subtracting $s$ from one of the parts
	$\rho_i$ of $\rho$; moreover, this part $\rho_i=r$ 
	should be $\ge s+1$. 
	This gives rise to the terms
	\begin{equation*}
		-\sum_{\textstyle\genfrac{}{}{0pt}{}{r\ge2}{1\le s\le r-1}}
		(2s-1)\gh_s\g_{r-s}\frac{\partial}{\partial\g_r},
	\end{equation*}

	Finally, we get
	\begin{equation*}
		\sum_{\si\colon\ell(\si)=1}D_\si
		=-\sum_{r\ge1}(2r-1)\gh_r\frac\partial{\partial\g_r}
		-\sum_{\textstyle\genfrac{}{}{0pt}{}{r\ge2}{1\le s\le r-1}}
		(2s-1)\gh_s\g_{r-s}\frac\partial{\partial\g_r}.
	\end{equation*}
	It remains to express $\gh_k$'s above in terms of $\g_i$'s
	using Proposition \ref{p3.4}. 
	In the first sum we should do this as
	\begin{equation*}
		\gh_r\to\g_r-\g_{r-1}\g_1-\dots-\g_1\g_{r-1}{}+{}
		\mbox{terms of degree $\le 2r-2$},\qquad r=1,2,\dots,
	\end{equation*}
	and in the second sum as
	\begin{equation*}
		\gh_s\to\g_s{}+{}\mbox{terms of degree $\le 2s-1$},\qquad s=1,2,\dots.
	\end{equation*}
	It is clear that this substitution affects only negligible terms
	(that is, summands of degree $\le-2$ in the expression for the operator $D$).
	
	To conclude the proof it remains to perform a simple transformation.
\end{proof}
Theorem \ref{p4.2} (1) now follows from
Lemmas \ref{p4.9}, \ref{p4.10} and \ref{p4.11}.

\subsection{The operator $U$ in differential form}\label{s4.5}
Here we prove claim (2)
of Theorem \ref{p4.2}, that is, compute the top degree terms
of the operator $U\colon\Gamma\to\Gamma$.
The argument here is similar to that of the previous subsection.
However, there are two differences which
require us to go through the proof in full detail:

$\bullet$ There is a difference between the expressions for $U\g_\rho$ and $D\g_\rho$ (Lemma \ref{p4.6});

$\bullet$ The behaviour of the degree of $\langle(x(x+1))^{m}\rangle^{\uparrow}=\g_m$ differs from the behaviour
of the degree of $\langle(y(y+1))^{m}\rangle^{\downarrow}=\gh_{m+1}$.
Indeed, the expression $\deg\langle(y(y+1))^{m}\rangle^{\downarrow}=2m+1$
is valid for all $m\in\Nn$, while
the formula $\deg\langle(x(x+1))^{m}\rangle^{\uparrow}=2m-1$ is valid only for $m>0$.

It is convenient to decompose $U$ using Lemma \ref{p4.6} as follows:
\begin{equation*}
	U=\sum_\si(U_\si^0+U_\si^1),
\end{equation*}
where the sum is taken over all ordinary partitions $\si$,
\begin{equation}\label{f69}
	U_\si^0\g_\rho:=
	\frac\al2
	\sum_{\tau\colon|\tau|=|\rho|-|\si|}c^{\rho}_{\si\tau}
	\big\langle a_\si\big(x(x+1)\big)\big\rangle^\uparrow\g_\tau
\end{equation}
and
\begin{equation}\label{f70}
	U_\si^1\g_\rho:=
	\frac12
	\sum_{\tau\colon|\tau|=|\rho|-|\si|}c^{\rho}_{\si\tau}
	\big\langle x(x+1)\cdot a_\si\big(x(x+1)\big)\big\rangle^\uparrow\g_\tau.
\end{equation}

\begin{rmk}\rm{}\label{p8.10}
	This decomposition $U=U^0+U^1$ is similar to the decomposition of the corresponding
	operator $U_{\theta,z,z'}=U^0_{\theta,z,z'}+U^1_{\theta,z,z'}+U^2_{\theta,z,z'}$ 
	in the paper \cite{Olshanski2009}. 
	The operator $U_{\theta,z,z'}$ is constructed using the up
	transition function
	for the $z$-measures on the Young graph with the Jack edge multiplicities
	(here $\theta>0$ is the Jack parameter), see \cite[Thm. 6.1 (ii)]{Olshanski2009}\footnote{Our
	operator $U$ is related to the multiplicative measures on the Schur 
	graph in the same way, see Lemma \ref{p2.10}.}.
	
	Note that our $U$ is expressed as the sum of two operators while 
	the expression for $U_{\theta,z,z'}$
	has three terms. 
	This is because
	the expression \cite[(4.3)]{Olshanski2009} for the up
	transition function for the $z$-measures involves
	terms of degrees zero, one and two in the (anisotropic) Kerov coordinates
	while the expression (\ref{f7}) for 
	the up transition function for the multiplicative measures
	involves only terms of degrees zero and one 
	in $x(x+1)$.\footnote{See 
	also \cite[Lemma 5.11]{Olshanski2009} and Proposition \ref{p3.7} in the present
	paper, respectively.}
\end{rmk}

\begin{lemma}[Cf. Lemma \ref{p4.7}]\label{p4.12}
	If $\si$ is a nonempty ordinary partition, then
	\begin{equation*}
		\deg U_\si^1\le
		\max_{\rho,\tau}(\ell(\rho)-\ell(\tau)-2\ell(\si)+1),
	\end{equation*}
	where the maximum is taken over all pairs $(\rho,\tau)$
	such that $c^{\rho}_{\si\tau}\ne0$.

	A more rough but simpler estimate is
	\begin{equation*}
		\deg U_\si^1\le-\ell(\si)+1.
	\end{equation*}
\end{lemma}
The case $U^0_\si$ will be investigated
separately, see Lemma \ref{p4.14} below.
\begin{proof}
	Arguing as in Lemma \ref{p4.7}, we get the estimate
	\begin{equation*}
		\deg U^1_\si\le
		\max_{\rho,\tau}\left( \Big\langle
		x(x+1)\cdot a_\si\big(x(x+1)\big)\Big\rangle^\uparrow 
		-2|\si|-\ell(\tau)+\ell(\rho)\right).
	\end{equation*}
	It remains to compute
	$\deg\left\langle x(x+1)\cdot a_\si\big(x(x+1)\big)\right\rangle^\uparrow$.
	Observe that the polynomial $x(x+1)\cdot a_\si\big(x(x+1)\big)$
	has degree $\ge1$ in $x(x+1)$, therefore,
	by Lemma \ref{p4.5} and  Proposition \ref{p3.7} we have
	\begin{equation*}
		\deg\Big\langle
		x(x+1)\cdot a_\si\big(x(x+1)\big)
		\Big\rangle^\uparrow=
		\deg\Big\langle
		\big(x(x+1)\big)^{|\si|-\ell(\si)+1}
		\Big\rangle^\uparrow
		=2|\si|-2\ell(\si)+1,
	\end{equation*}
	this gives the first estimate.
	The second estimate is obtained
	as before,
	because if $c^{\rho}_{\si\tau}\ne0$, then
	$\ell(\rho)\le\ell(\si)+\ell(\tau)$.
\end{proof}
From the second estimate of the above lemma follows a
\begin{corollary}[Cf. Corollary \ref{p4.8}]\label{p4.13}
	If $\ell(\si)\ge3$, then 
	$\deg U_\si^1\le-2$.
\end{corollary}
In the next Lemma we deal with the whole operator $\sum_{\si}U_\si^0$.
\begin{lemma}\label{p4.14}
	\begin{equation*}
		\sum_\si U_\si^0=\frac\al2+\al\frac{\partial}{\partial\g_1}{}+{}
		\text{\rm{}operators of degree $\le-2$}.
	\end{equation*}
\end{lemma}
\begin{proof}
	If $\si$ is empty, then $\rho=\tau$ and $c^{\rho}_{\si\tau}=1$ in (\ref{f69}), and we get
	\begin{equation*}
		U^0_\si\g_\rho=\frac\al2 c^{\rho}_{\si\tau}
		\left\langle 1\right\rangle^\uparrow\g_\rho=\frac\al2\g_\rho
	\end{equation*}
	(compare this to Lemma \ref{p4.9}).

	Assume now that $\si$ is nonempty. Then
	$\si=(\si_1,\dots,\si_{\ell(\si)})$, and $\ell(\si)\ge1$.
	Arguing as in Lemma \ref{p4.12}, we get the estimate
	\begin{equation*}
		\begin{array}{l}
			\displaystyle
			\deg U^0_\si\le\max_{\rho,\tau}
			\left( \deg\left\langle a_\si\big(x(x+1)\big)\right\rangle^\uparrow 
			+2|\tau|-\ell(\tau)-2|\rho|+\ell(\rho)\right)
			\\
			\displaystyle\qquad
			\le
			\max_{\rho,\tau}\left( \deg\left\langle
			a_\si\big(x(x+1)\big)\right\rangle^\uparrow
			-2|\si|+\ell(\si)\right),
		\end{array}
	\end{equation*}
	where the maximum is taken over all pairs $(\rho,\tau)$
	such that $c^{\rho}_{\si\tau}\ne0$.
	The second inequality holds because 
	$\ell(\rho)-\ell(\tau)\le\ell(\si)$.

	If the polynomial $a_\si\big(x(x+1)\big)$
	has degree $\ge1$ in $x(x+1)$, 
	then using Lemma \ref{p4.5} and Proposition \ref{p3.7} we can write
	$$
		\deg\left\langle a_\si\big(x(x+1)\big)\right\rangle^\uparrow
		=\deg\left( \g_{|\si|-\ell(\si)} \right)=
		2|\si|-2\ell(\si)-1,
	$$ 
	which implies 
	\begin{equation*}
		\deg U^0_\si\le-\ell(\si)-1\le-2.
	\end{equation*}
	Thus, it remains to consider the case when
	$\si$ is nonempty and the polynomial
	$a_\si\big(x(x+1)\big)$ has zero degree in $x(x+1)$.
	This case occurs if and only if $\si_1=\dots=\si_{\ell(\si)}=1$, and for 
	$\deg U_\si^0$ we have the estimate:
	$\deg U^0_\si\le-\ell(\si)$.
	
	If $\ell(\si)\ge2$, this is enough to conclude $\deg U^0_\si\le-2$.
	
	Finally, examine the case $\si=(1)$. There are two possibilities:
	$\ell(\rho)=\ell(\tau)+1$ and $\ell(\rho)=\ell(\tau)$.
	In the latter case the estimate of $\deg U^0_\si$
	can be refined because then $\ell(\rho)-\ell(\tau)=0$
	is strictly smaller than $\ell(\si)=1$, which again implies $\deg U^0_\si\le-2$.

	Thus, the only substantial contribution arises when $\si=(1)$,
	$\rho_{\ell(\rho)}=1$ and
	and $\tau=(\rho_1,\dots,\rho_{\ell(\rho)-1})$
	(that is, $\tau$ is obtained from $\rho$ by deleting a singleton).
	This gives rise to the term
	$U^0_{\si}\g_\rho=
	\al{\partial\g_\rho}/{\partial\g_1}$.
\end{proof}
Lemmas \ref{p4.15}, \ref{p4.16} and \ref{p4.17} below are similar to Lemmas 
\ref{p4.9}, \ref{p4.10} and \ref{p4.11}, respectively, and
deal with the operator $\sum_{\si}U^1_\si$.\footnote{It follows from Corollary \ref{p4.13} that it suffices 
to examine the cases when $\ell(\si)=0$ (that is, $\si=\varnothing$),
$\ell(\si)=2$, and $\ell(\si)=1$. We perform this consequently.}
\begin{lemma}\label{p4.15}
	$U_\varnothing^1=\displaystyle\frac12\g_1$.
\end{lemma}
\begin{proof}
	The proof is similar to that of Lemma \ref{p4.9}.
	We should only note that
	$$
		\left\langle x(x+1)\cdot a_{\varnothing}\big(x(x+1)\big)\right\rangle^\uparrow=
		\big\langle x(x+1)\big\rangle^\uparrow=\g_1
	$$
	in (\ref{f70}), therefore, $U_\varnothing^1$ reduces to multiplication by $\g_1/2$.
\end{proof}
\begin{lemma}\label{p4.16}
	\begin{equation*}
		\sum_{\si\colon\ell(\si)=2}U_\si^1=
		\sum_{r,s\ge1}(2r-1)(2s-1)
		\g_{r+s-1}\frac{\partial^2}{\partial\g_r\partial\g_s}
		{}+{}\text{\rm{}operators of degree $\le-2$}.
	\end{equation*}
\end{lemma}
\begin{proof}
	The proof
	is similar to the proof of Lemma \ref{p4.10}.
	Instead of Lemma \ref{p4.7}, we refer to 
	its analogue, Lemma \ref{p4.12}.

	Let $\ell(\si)=2$ in (\ref{f70}),
	that is, $\si=(\si_1,\si_2)$ with $\si_1\ge\si_2\ge1$.
	Using Lemma \ref{p4.5} and Proposition \ref{p3.7}, we get
	\begin{equation*}
		\begin{array}{l}\displaystyle
			\Big\langle
			x(x+1)\cdot a_\si\big(x(x+1)\big)
			\Big\rangle^\uparrow
			\\\rule{0pt}{16pt}\displaystyle\qquad
			=4(2\si_1-1)(2\si_2-1)\g_{\si_1+\si_2-1}{}+{}\mbox{terms of degree $\le2|\si|-5$}.
		\end{array}
	\end{equation*}
	Next,
	\begin{equation*}
		\begin{array}{rcl}\displaystyle
			\left( \sum_{\si\colon\ell(\si)=2}U_\si^1 \right)\g_\rho&=& \displaystyle
			\sum_{1\le i<j\le\ell(\rho)}2(2\rho_i-1)(2\rho_j-1)c^{\rho}_{\si\tau}\g_{\rho_i+\rho_j-1}\g_{\rho\setminus\left\{ \rho_i,\rho_j \right\}}
			\\&&\qquad\displaystyle{}+{}\mbox{terms of degree $\le2|\rho|-\ell(\rho)-3$},
		\end{array}
	\end{equation*}
	and using (\ref{f68}), we conclude the proof.
\end{proof}
\begin{lemma}\label{p4.17}
	\begin{equation*}
		\begin{array}{rcl}\displaystyle
			\sum_{\si\colon\ell(\si)=1}U_\si^1&=& \displaystyle
			\sum_{r\ge1}(2r-1)\g_r\frac{\partial}{\partial\g_r}
			+\sum_{r,s\ge1}(r+s-1)\g_r\g_s\frac{\partial}{\partial\g_{r+s}}\\&&\displaystyle\qquad
			{}+{}\text{\rm{}operators of degree $\le-2$}.\qquad
		\end{array}
	\end{equation*}
\end{lemma}
\begin{proof}
	The proof is similar to that of Lemma \ref{p4.11}
	(and we again use Lemma \ref{p4.12} instead of Lemma \ref{p4.7}).
	
	Let $\ell(\si)=1$ in (\ref{f70}), that is, $\si=(s)$ for some $s\in\N$.
	Again, two cases are possible: either
	$\ell(\rho)=\ell(\tau)+1$, or $\ell(\rho)=\ell(\tau)$.

	Assume first that $\ell(\rho)=\ell(\tau)+1$. Using Lemma \ref{p4.5} and Proposition \ref{p3.7}, 
	we get
	\begin{equation*}
		\Big\langle
		x(x+1)\cdot a_\si\big(x(x+1)\big)
		\Big\rangle^\uparrow=
		2(2s-1)\g_s{}+{}\mbox{terms of degree $\le 2s-3$}.
	\end{equation*}
	This gives rise to the terms
	\begin{equation*}
		\sum_{r\ge1}(2r-1)\g_r\frac\partial{\partial\g_r}{}+{}\text{operators of degree $\le-2$}.
	\end{equation*}

	If $\ell(\rho)=\ell(\tau)$, similarly to the proof of Lemma \ref{p4.11},
	we get the terms
	\begin{equation*}
		\sum_{\textstyle\genfrac{}{}{0pt}{}{r\ge2}{1\le s\le r-1}}
		(2s-1)\g_s\g_{r-s}\frac\partial{\partial\g_r}{}+{}
		\text{operators of order $\le-2$}.
	\end{equation*}
	
	To conclude the proof it remains to perform a simple transformation.
\end{proof}
Theorem \ref{p4.2} (2) now follows from
Lemmas \ref{p4.14}, \ref{p4.15}, \ref{p4.16} and \ref{p4.17}.

\section{The operator $T_n$ in differential form}\label{s5}
Let $T_n\colon\mathrm{Fun}(\Sb_n)\to\mathrm{Fun}(\Sb_n)$, $n\in\N$,
be the operator from Definition \ref{p10.6}.
In \S\ref{s2} we have obtained a formula for the action of $T_n$ on Schur's $\mQ$-functions
(Theorem \ref{p2.7}).
In this section using the results 
of \S\ref{s3} and 
Theorem \ref{p4.2} we prove another formula for $T_n$:
\begin{thm}\label{p5.1}
	There exists a unique operator $\widetilde B\colon\Gamma\to\Gamma$ such that
	\begin{equation*}
		(T_n-\mathbf{1})f_n=\frac{(\widetilde Bf)_n}{(n+\al/2)(n+1)}
	\end{equation*}
	for all $f\in\Gamma$.
	
	The operator $\widetilde B$ has zero degree. Under the identification of $\Gamma$ with the polynomial 
	algebra $\mathbb{R}\left[ p_1,p_3,p_5,\dots \right]$, the zero-degree
	homogeneous component of $\widetilde B$, 
	the operator $B\colon\Gamma\to\Gamma$, has the form
	\begin{equation}\label{f75}
		\left.
		\begin{array}{l}
			\displaystyle
			B=\sum_{i,j=2}^{\infty}(2i-1)(2j-1)
			(p_1p_{2i+2j-3}-p_{2i-1}p_{2j-1})\frac{\partial^2}{\partial p_{2i-1}\partial p_{2j-1}}
			\\\displaystyle\qquad
			+2\sum_{i,j=1}^{\infty}(2i+2j-1)p_1p_{2i-1}p_{2j-1}\frac{\partial}{\partial p_{2i+2j-1}}
			\\\displaystyle\qquad
			-\sum_{i=2}^{\infty}(2i-1)\left(2i-2+\frac\al2\right)p_{2i-1}\frac{\partial}{\partial p_{2i-1}}.
		\end{array}
		\right.
	\end{equation}
\end{thm}
By the zero-degree homogeneous component of the operator $\widetilde B$
we mean the unique homogeneous operator $B\colon\Gamma\to\Gamma$ of zero degree
such that 
$$
	\widetilde B=B+\mbox{operators of degree $\le-1$}.
$$

First, we note an important corollary of Theorem \ref{p5.1}:
\begin{corollary}\label{p5.2}
	The operator $B\colon\Gamma\to\Gamma$ commutes with the operator of multiplication
	by the element $p_1\in\Gamma$.
\end{corollary}
\begin{proof}
	This follows from the fact that the expression (\ref{f75}) for $B$
	does not contain partial derivatives with respect to $p_1$.
\end{proof}	
In the rest of this section we prove Theorem \ref{p5.1}.

First, (\ref{f41}) and (\ref{f54}) imply
\begin{equation*}
	(T_n-{\bf1})f_n=\frac{\big( (UD-\frac14(\g_1+\al)(\g_1+2))f\big)_n}{(n+\al/2)(n+1)}
\end{equation*}
for all $f\in\Gamma$.
Thus, $\widetilde B=UD-\frac14(\g_1+\al)(\g_1+2)$, and the uniqueness of $\widetilde B$ follows from the 
fact that the algebra $\Gamma\subset\mathrm{Fun}(\Sb)$ separates points.

Now using Theorem \ref{p4.2} 
we write the operator $\widetilde B$ as a formal differential operator with respect
to the generators $\g_k$, $k\in\N$, of the algebra $\Gamma$:
\begin{lemma}\label{p5.3}
	Under the identification of the algebra $\Gamma$ 
	with the polynomial algebra $\mathbb{R}\left[ \g_1,\g_2,\dots \right]$,
	the operator $\widetilde B=UD-\frac14(\g_1+\al)(\g_1+2)$ looks as follows:
	\begin{equation*}
		\left.
		\begin{array}{l}\displaystyle
			\widetilde B=
			\sum_{r,s=2}^{\infty}(2r-1)(2s-1)(\g_1\g_{r+s-1}-\g_r\g_s)\frac{\partial^2}{\partial\g_r\partial\g_s}
			\\\displaystyle\qquad
			+\sum_{r,s=1}^{\infty}(r+s-1/2)\g_1\g_r\g_s\frac{\partial}{\partial\g_{r+s}}\\\displaystyle\qquad
			-
			\sum_{r=2}^{\infty}
			(2r-1)\left(2r-2+\frac\al2\right)\g_r\frac{\partial}{\partial\g_r}

			{}+{}\mbox{\rm{}operators of degree $\le-1$}.
		\end{array}
		\right.
	\end{equation*}
\end{lemma}
Note that the operators 
$U$ and $D$ both have degree $1$
in the sense of Definition \ref{p4.1}.
However, it turns out that the operator $\widetilde B$
has zero degree instead of degree $2$, because higher degree terms cancel out.
\begin{proof}
	We write
	\begin{equation*}
		U=\frac12\g_1+U_0+U_{-1}+\dots,\qquad
		D=\frac12\g_1+D_0+D_{-1}+\dots,
	\end{equation*}
	where dots stand for operators of degree $\le-2$,
	\begin{equation*}
		U_0:=\frac12\al+\sum_{r=1}^{\infty}(2r-1)\g_r\frac\partial{\partial\g_r},
		\qquad 
		D_0:=-\sum_{r=1}^{\infty}(2r-1)\g_r\frac\partial{\partial\g_r}
	\end{equation*}
	(these are the zero degree parts), 
	\begin{equation*}
		\begin{array}{l}
			\displaystyle
			U_{-1}:=\al\frac\partial{\partial\g_1}+\sum_{r,s=1}^{\infty}
			(2r-1)(2s-1)\g_{r+s-1}\frac{\partial^2}{\partial\g_r\partial\g_s}
			\\\displaystyle\qquad
			+\sum_{r,s=1}^{\infty}(r+s-1)\g_r\g_s\frac\partial{\partial\g_{r+s}},\\
			\displaystyle
			D_{-1}:=\sum_{r,s=1}^{\infty}(2r-1)(2s-1)\g_{r+s-1}\frac{\partial^2}{\partial\g_r\partial\g_s}
			+\sum_{r,s=1}^{\infty}(r+s)\g_r\g_s\frac\partial{\partial\g_{r+s}}
		\end{array}
	\end{equation*}
	(these are the parts of degree $-1$).
	
	We compute the top degree 
	terms of the operator $\widetilde B=UD-\frac14(\g_1+\al)(\g_1+2)$
	consequently.

	Terms of degree $2$:
	\begin{equation*}
		\frac14\g_1^2-\frac14\g_1^2=0.
	\end{equation*}

	Terms of degree $1$ are equal to
	\begin{equation*}
		\frac12\g_1D_0+\frac12U_0\g_1-\frac14(\al+2)\g_1.
	\end{equation*}
	Because the operator $\sum_{r=2}^{\infty}(2r-1)\g_r\frac{\partial}{\partial\g_r}$ commutes
	with the multiplication by $\g_1$, we have
	\begin{equation*}
		\left.
		\begin{array}{l}
			\displaystyle
			\g_1D_0+U_0\g_1=-\g_1
			\left( \g_1\frac{\partial}{\partial\g_1}+
			\sum_{r=2}^{\infty}(2r-1)\g_r\frac{\partial}{\partial\g_r} \right)
			\\\displaystyle\qquad\qquad	
			+\left( \frac\al2+\g_1\frac{\partial}{\partial\g_1}+ 
			\sum_{r=2}^{\infty}(2r-1)\g_r\frac{\partial}{\partial\g_r} \right)\g_1\\\displaystyle\qquad=
			-\g_1^2\frac{\partial}{\partial\g_1}+\frac\al2\g_1+\g_1\frac{\partial}{\partial\g_1}\g_1
			=\frac\al2\g_1+\g_1,
		\end{array}
		\right.
	\end{equation*}
	and we see that the terms of degree $1$ also cancel out.

	It remains to compute terms of degree $0$. They are equal to
	\begin{equation*}
		\frac12\g_1D_{-1}+U_0D_0+\frac12U_{-1}\g_1-\frac\al2.
	\end{equation*}
	Observe that
	\begin{equation*}
		\left.
		\begin{array}{l}
			\displaystyle
			\left[ U_{-1},\g_1 \right]=
			\left[ \al\frac{\partial}{\partial\g_1},\g_1 \right]+
			\left[ \g_1\frac{\partial^2}{\partial\g_1^2},\g_1 \right]
			+
			2\sum_{r=2}^{\infty}(2r-1)\g_r\frac{\partial}{\partial\g_r}
			\left[ \frac{\partial}{\partial\g_1},\g_1 \right]\\\displaystyle\qquad=
			\al+
			2\g_1\frac{\partial}{\partial\g_1}
			+2\sum_{r=2}^{\infty}(2r-1)\g_r\frac{\partial}{\partial\g_r}.
		\end{array}
		\right.
	\end{equation*}
	It can be readily verified that
	\begin{equation*}
		\begin{array}{l}\displaystyle
			U_0D_0\\\displaystyle\quad=-\left( \frac\al2+\g_1\frac{\partial}{\partial\g_1}+\sum_{r=2}^{\infty} 
			(2r-1)\g_r\frac{\partial}{\partial\g_r}\right)
			\left( \g_1\frac{\partial}{\partial\g_1}+
			\sum_{r=2}^{\infty}(2r-1)\g_r\frac{\partial}{\partial\g_r}\right)
			\\\displaystyle\quad=
			-\frac12\al\g_1\frac{\partial}{\partial\g_1}-
			           \g_1\frac{\partial}{\partial\g_1}-
				   \g_1^2\frac{\partial^2}{\partial\g_1^2}
     			-2\sum_{r=2}^{\infty}(2r-1)\g_1\g_r\frac{\partial^2}{\partial\g_1\partial\g_r}
			\\\displaystyle\quad\qquad
			-\frac\al2\sum_{r=2}^{\infty}
			(2r-1)\g_r\frac{\partial}{\partial\g_r}-
			\sum_{r,s=2}^{\infty}(2r-1)(2s-1)\g_r\g_s\frac{\partial^2}{\partial\g_r\partial\g_s}\\\displaystyle
			\qquad\quad-
			\sum_{r=2}^{\infty}(2r-1)^2\g_r\frac{\partial}{\partial\g_r}
		\end{array}
	\end{equation*}
	and
	\begin{equation*}
		\left.
		\begin{array}{l}\displaystyle
			\frac12(D_{-1}+U_{-1})=\frac12\al\frac{\partial}{\partial\g_1}+
			\sum_{r,s=2}^{\infty}(2r-1)(2s-1)\g_{r+s-1}\frac{\partial^2}{\partial\g_r\partial\g_s}
			\\\displaystyle\qquad+
			\g_1\frac{\partial^2}{\partial\g_1^2}+2\sum_{r=2}^{\infty}(2r-1)\g_r\frac{\partial^2}{\partial\g_r\partial\g_1}
			+\sum_{r,s=1}^{\infty}(r+s-1/2)\g_r\g_s\frac{\partial}{\partial\g_{r+s}}.
		\end{array}
		\right.
	\end{equation*}
	Now we finally are able to compute the terms of degree $0$:
	\begin{equation*}
		\begin{array}{l}\displaystyle
			\frac12\g_1D_{-1}+U_0D_0+\frac12U_{-1}\g_1-\frac\al2
			\\\displaystyle\qquad=
			\frac12\g_1(D_{-1}+U_{-1})+U_0D_0+
			\g_1\frac{\partial}{\partial\g_1}
			+\sum_{r=2}^{\infty}(2r-1)\g_r\frac{\partial}{\partial\g_r}.
		\end{array}
	\end{equation*}
	Combining three above formulas, we get the desired expression.
\end{proof}

To prove Theorem \ref{p5.1}
it remains to substitute 
in the expression for $\widetilde B$ given by the previous Lemma
the inhomogeneous
generators $\g_k$, $k\in\N$, 
by the homogeneous generators $p_{2m-1}$, $m\in\N$.
This should be done according to the next Lemma:
\begin{lemma}\label{p5.4}
	{\rm{}(1)\/} $\g_k=2p_{2k-1}+{}\mbox{\rm{}terms of degree $\le(2k-1)$}$, \ $k\in\N$;

	{\rm{}(2)\/} Let $f\in\Gamma$, then\footnote{Note 
	that both $\pd f/\pd \g_k$ and $\pd f/\pd p_{2k-1}$ have degree $\big(\deg f-(2k-1)\big)$.}
	\begin{equation*}
		\frac{\pd f}{\pd \g_k}=\frac12\frac{\pd f}{\pd p_{2k-1}}+{}\mbox{\rm{}terms of degree $\le\big(\deg f-(2k-1)\big)$},\qquad
		k\in\N.
	\end{equation*}
\end{lemma}	
\begin{proof} Claim (1) directly follows from Propositions \ref{p3.3} and \ref{p3.4}.

	To prove claim (2) observe that
	\begin{equation*}
		\left.
		\begin{array}{c}
			\displaystyle
			\frac{\pd f}{\pd \g_k}=
			\sum_{l\ge k}\frac{\pd(2p_{2l-1})}{\pd \g_k}\frac{\pd f}{\pd(2p_{2l-1})}=
			\frac12\frac{\pd f}{\pd p_{2k-1}}+\sum_{l>k}\frac{\pd(2p_{2l-1}-\g_l)}{\pd \g_k}
			\frac{\pd f}{\pd(2p_{2l-1})}.
		\end{array}
		\right.
	\end{equation*}
	The $l$th summand in the last sum has degree
	\begin{equation*}
		\le (2l-3)-(2k-1)+\deg f-(2l-1)=\deg f-(2k-1)-2<\deg f-(2k-1).
	\end{equation*}
	This concludes the proof.
\end{proof}	
Together Lemmas \ref{p5.3} and \ref{p5.4} imply Theorem \ref{p5.1}.

\section{The limit diffusion}\label{s6}
In this section we prove that the Markov chains $T_n$
from Definition \ref{p10.6}
converge (within a certain time scaling) to a continuous time
Markov process $\X{\al}(t)$, $t\ge0$, in the simplex $\sk$.
Using Theorem \ref{p5.1} we prove the expression (\ref{f0.1})
(from Introduction)
for its the pre-generator.
In \S\ref{s70.7}
we study some further properties of $\X{\al}(t)$.

In \S\ref{s8.1}
we discuss the embedding of
the simplex $\sk$ into the Thoma simplex~$\Omega$
introduced in \cite{GnedinIntern.Math.ResearchNotices2006Art.ID5196839pp.}
(see also \S\ref{s0.2}). This construction leads to
another proof of one of the results
from this section (namely, the first claim of Proposition \ref{p6.6})
but it is also of separate interest.

\subsection{An operator semigroup approximation theorem}\label{s6.1}
We begin by stating a well-known general result on approximations of continuous contraction semigroups by
discrete ones.
We formulate it in a form (Theorem~\ref{p6.4}) best suitable
for the application to our concrete situation.
We refer to the paper \cite{Trotter1958} 
and the book \cite{Ethier1986}. 
In the book one can also find additional 
references.

Let $L$ and $L_n$, $n\in\N$, be real Banach spaces.\footnote{The norms
of $L$ and of each $L_n$ are denoted by the same symbols $\|\cdot\|$.}
Let $\pi_n\colon L\to L_n$, $n\in\N$, 
be bounded linear operators such that $\sup_n\|\pi_n\|<\infty$.
\begin{df}\rm{}\label{p6.1}
	We say that a sequence of elements 
	$\left\{ f_n\in L_n \right\}$ {\em{}converges\/} to an element $f\in L$ if
	$\lim\limits_{n\to\infty}\|\pi_n f-f_n\|=0$.
	We write $f_n\to f$.
\end{df}	
It our concrete situation described below in \S\ref{s6.2} the additional condition
\begin{equation}\label{f80}
	\lim_{n\to\infty}\|\pi_n f\|=\|f\|\qquad\mbox{for all $f\in L$}
\end{equation}
is satisfied. This condition implies that any 
sequence $\left\{ f_n\in L_n \right\}$ may have at most one limit in $L$.

\begin{df}\rm{}\label{p6.3}
	An operator $D$ in $L$ is called {\em{}dissipative\/}
	if $\|(s\mathbf{1}-D)f\|\ge s\|f\|$ for all $s\ge0$ and all $f$ from the domain of $D$,
	where ${\bf1}$ denotes the identity operator.
\end{df}

Now, assume that for all $n\in\N$ we are given
a contraction operator $T_n$ in $L_n$.
Suppose that 
$\{\varepsilon_n\}$ is a sequence 
of positive numbers converging to 
zero.
Assume that there exists a dense subspace $\Fc\subset L$ and an operator $A\colon \Fc\to L$
such that in the sense of Definition \ref{p6.1}
\begin{equation*}
	\varepsilon_n^{-1}(T_n-\mathbf{1})\pi_nf\to Af
	\qquad\mbox{for all $f\in\Fc$}.
\end{equation*}
\begin{thm}\label{p6.4}
	If 

	$\bullet$ The operator $A\colon\Fc\to L$ is dissipative;
	
	$\bullet$ For some $s>0$ the range of $(s\mathbf{1}-A)$ is dense in $L$,

	Then the operator $A$ is closable in $L$; its 
	closure generates a strongly continuous
	contraction semigroup $\left\{ T(t) \right\}_{t\ge0}$ in $L$;
	and (again in the sense of Definition \ref{p6.1}) 
	\begin{equation}\label{f81}
		T_n^{[\varepsilon_n^{-1}t]}\pi_nf\to T(t)f\qquad\mbox{for all $f\in L$},
	\end{equation}
	for $t\ge0$ uniformly on bounded intervals.
\end{thm}
\begin{proof}
	Let $\hat A\colon L\to L$ be the operator defined as
	$f\mapsto \lim\limits_{n\to\infty}\varepsilon_n^{-1}(T_n-\mathbf{1})\pi_nf$
	for all $f\in L$ such that this limit exists in the 
	sense of Definition \ref{p6.1}. The domain of 
	$\hat A$ consists of such $f\in L$. Clearly, $\hat A\mid_{\Fc}=A$.

	Since each $T_n$ is a contraction, each operator $\varepsilon_n^{-1}(T_n-{\bf1})$ is dissipative.
	Hence $\hat A$ is dissipative, too.

	The operator $\hat A$ satisfies the conditions of \cite[Thm. 5.3]{Trotter1958} with $M=1$ and $K=0$.
	Hence $\hat A$ is closable, its closure $\overline{\hat A}$ generates a semigroup $\left\{ T(t) \right\}_{t\ge0}$
	in $L$, and the convergence (\ref{f81}) holds pointwise (with respect to $t$).
	The uniform convergence follows from the implication 
	(b)$\Rightarrow$(a) of \cite[Ch. 1, Thm. 6.5]{Ethier1986}.

	Because each operator $T_n$ is a contraction, 
	the fact that each 
	$T(t)$ is a contraction follows
	from (\ref{f80}) and (\ref{f81}). 
	By the Hille-Yosida theorem 
	(see, e.g., \cite[Ch. 1,Thm. 2.6]{Ethier1986}), 
	the dissipativity of $\hat A$ implies 
	that the semigroup $T(t)$ is strongly continuous.

	Since $\Fc$ is dense in $L$ and for some $s>0$
	the range of $s\mathbf{1}-\overline{\hat A}\mid_\Fc=s\mathbf{1}-A$
	is dense in $L$, the subspace $\Fc\subset L$
	is a {\em{}core\/} for $\overline{\hat A}$ in the sense of \cite[Ch. 1, Sect. 3]{Ethier1986}.
	By \cite[Ch. 1, Prop. 3.1]{Ethier1986}, the operator $A$ is closable in $L$
	and $\overline{\hat A}=\overline A$. This concludes the proof.
\end{proof}

\subsection{The simplex $\sk$}\label{s6.2}
We return to our concrete situation. 
As $L_n$, $n\in\N$, we take the finite-dimensional vector space 
$\mathrm{Fun}(\Sb_n)$ of real-valued 
functions on $\Sb_n$
with the supremum norm.
As $T_n$ we take the Markov transition operators from Definition \ref{p10.6}.
Clearly, each $T_n$ is a contraction.
As the scaling factors we take $\varepsilon_n:=1/n^2$.
To define the space $L$ and the operators $\pi_n\colon L\to L_n$ we
need some extra notation.

Let $\sk$ be the subset of the infinite-dimensional cube $\left[ 0,1 \right]^{\infty}$
defined as
\begin{equation*}
	\sk:=\left\{ \x=(\x_1,\x_2,\dots)\in\left[ 0,1 \right]^{\infty}\colon \x_1\ge\x_2\ge\dots\ge0,\ \sum_{i}\x_i\le1 \right\},
\end{equation*}
We equip the cube $\left[ 0,1 \right]^{\infty}$ with the standard product topology.
The subset~$\sk\subset\left[ 0,1 \right]^{\infty}$ is a compact, metrizable and separable space. 
As $L$ we take the Banach space $C(\sk)$ of all real continuous functions on $\sk$ with pointwise operations and the supremum norm.

For $n\in\N$, we define an embedding $\iota_n$
of the set $\Sb_n$ into the space $\sk$:
\begin{equation}\label{f81.1}
	\iota_n\colon\Sb_n\hookrightarrow\sk,\qquad \la=(\la_1,\dots,\la_{\ell},0,0,\dots)\mapsto
	\left( \frac{\la_1}n,\dots,\frac{\la_{\ell}}n,0,0,\dots \right)\in\sk.
\end{equation}
Using $\iota_n$ we define the operators $\pi_n\colon L\to L_n$, that is, $\pi_n\colon C(\sk)\to\mathrm{Fun}(\Sb_n)$:
\begin{equation*}
	(\pi_nf)(\la):=f(\iota_n(\la)),\qquad\mbox{where $f\in C(\sk)$ and $\la\in\Sb_n$}.
\end{equation*}
Clearly, $\|\pi_n\|\le1$.
Moreover, in our situation the condition (\ref{f80}) is satisfied
because the space $\sk$ is approximated
by the sets $\iota_n(\Sb_n)\subset\sk$ in the sense that every open subset
of $\sk$ has a nonempty intersection with $\iota_n(\Sb_n)$ for all $n$ large enough.

\subsection{Moment coordinates}\label{s6.3}
Here we define the dense subspace $\Fc\subset L=C(\sk)$.

To every point $\x\in\sk$ we assign a probability measure
\begin{equation*}
	\nu_\x:=\sum_{i=1}^{\infty}\x_i\delta_{\x_i}+\gamma\delta_0,\qquad \gamma:=1-\sum_{i=1}^{\infty}\x_i
\end{equation*}
on $\left[ 0,1 \right]$, where by $\delta_s$ we denote the Dirac measure at a point $s$.
Denote by $\mq_k=\mq_k(\x)$ the $k$th moment of the measure $\nu_\x$:
\begin{equation*}
	\mq_k(\x):=\int_0^1u^k\nu_\x(du)=\sum_{i=1}^{\infty}\x_i^{k+1},\qquad k=1,2,\dots.
\end{equation*}
Following \cite{Borodin2007}, we call $\mq_1,\mq_2,\dots$ the {\em{}moment coordinates\/}
of the point $\x\in\sk$.
They are continuous functions on $\sk$.\footnote{Observe 
that the function $\x\mapsto\sum_{i=1}^{\infty}\x_i$ is not continuous in $\x\in\sk$.}

Note that the functions $\mq_1,\mq_2,\dots$ are algebraically independent as functions
on $\sk$. Clearly, any subcollection of $\left\{ \mq_1,\mq_2,\dots \right\}$ is also 
algebraically independent.
As $\Fc$ we take
the subalgebra of the Banach algebra $C(\sk)$ freely generated by the {\em{}even\/} moment coordinates:
\begin{equation*}
	\Fc:=\mathbb{R}\left[ \mq_2,\mq_4,\mq_6,\dots \right]\subset C(\sk).
\end{equation*}
\begin{prop}\label{p6.6}
	The functions $\mq_2,\mq_4,\mq_6,\dots$ separate points of $\sk$.
	Moreover, 
	any infinite subcollection
	of $\left\{ \mq_1,\mq_2,\dots \right\}$
	also possesses this property.
\end{prop}
\begin{proof}
	Let $\left\{ \mq_{k_1},\mq_{k_2},\dots \right\}$
	be any infinite subcollection of $\left\{ \mq_1,\mq_2,\dots \right\}$.
	It suffices to show that a point $\x\in\sk$
	is uniquely determined by the sequence
	$\left\{ \mq_{k_1}(\x),\mq_{k_2}(\x),\dots \right\}$. 
	Observe that for every $m=0,1,2,\dots$ we have
	\begin{equation*}
		\Big(\mq_{k_n}(\x)-\sum_{j=1}^{m}\x_j^{k_n}\Big)^{1/k_n}=
		\x_{m+1}\cdot\Big(1+\sum_{i=m+1}^{\infty}
		(\x_i/\x_{m+1})^{k_n}\Big)^{1/k_n}\to\x_{m+1}
	\end{equation*}
	as $n\to\infty$ (the convergence is pointwise). 
	Here if $m=0$, then by agreement there is no sum $\sum_{j=1}^{m}$ in the LHS.
	Using this convergence, one can reconstruct the 
	coordinates
	$\x_1,\x_2,\x_3,\dots$
	one after another using the sequence 
	$\left\{ \mq_{k_1}(\x),\mq_{k_2}(\x),\dots \right\}$.
	This concludes the proof.
\end{proof}
Since the subalgebra $\Fc\subset C(\sk)$ separates points and contains the function $1$, it is dense in $C(\sk)$.

Next, recall the algebra $\Gamma=\mathbb{R}\left[ p_1,p_3,p_5,\dots \right]$ of 
doubly symmetric functions
introduced in \S\ref{s2.1}. 
Let $I:=(p_1-1)\Gamma$ be the principal ideal in $\Gamma$ generated by $(p_1-1)$. Set $\Gamma^\circ:=\Gamma/I$.
To every element $f\in\Gamma$ corresponds an image in $\Gamma^\circ$ denoted by $f^\circ$.
In particular, $p_1^\circ=1$, and $\Gamma^\circ$ is freely generated
(as a commutative unital algebra) by the elements $p_3^\circ,p_5^\circ,p_7^\circ,\dots$. 

Observe also that 
\begin{equation}\label{f83}
	\Gamma=\mathbb{R}\left[ p_1,p_3,p_5,p_7,\dots \right]=I\oplus\mathbb{R}\left[ p_3,p_5,p_7,\dots \right],
\end{equation}
and hence $\Fc\cong\mathbb{R}\left[ p_3,p_5,p_7,\dots \right]$. We will use this fact below.

The correspondence 
\begin{equation*}
	p_{2k+1}^{\circ}\longleftrightarrow \mq_{2k},\qquad k=1,2,\dots
\end{equation*}
establishes an isomorphism between the algebras $\Gamma^\circ$ and $\Fc$.
We will identify elements
$g\in\Gamma^\circ$ and the corresponding continuous functions $g(\x)$ on $\sk$.
Moreover, to every element $\phi\in\Gamma$ corresponds a continuous function
on $\sk$. Denote this function by $\phi^\circ(\x)$.
Equivalently, the map $\phi\to\phi^\circ(\cdot)$ is determined by setting
\begin{equation*}
	p_1^\circ(\x):\equiv 1,\qquad
	p_{2k-1}^\circ(\x):=\sum_{i=1}^{\infty}\x_i^{2k-1},\qquad k=2,3,\dots.
\end{equation*}

\subsection{The limit theorem for coherent systems}\label{s6.4}
At this point it is convenient to formulate the following theorem
about coherent systems on $\Sb$. 

\begin{thm}\label{p6.7}
	Let $\left\{ M_n \right\}$ be a coherent system on $\Sb$ (see \S\ref{s1.3} for the definition).
	Then the push-forward of the measure $M_n$ under the embedding $\iota_n$ (defined in \S\ref{s6.2})
	weakly converges, as $n\to\infty$, to a probability measure $\mathsf{P}$ on $\sk$.
	The measure $\mathsf{P}$ is called the {\em{}boundary measure\/} of the system $\left\{ M_n \right\}$.

	Conversely, any coherent system on $\Sb$ can be reconstructed from its boundary measure as follows:
	\begin{equation*}
		M_n(\la)=2^{-|\la|}\dhh(\la)\int_{\sk}\mQ_\la^\circ(\x)\mathsf{P}(d\x)\qquad\mbox{for all $n\in\N$ and $\la\in\Sb_n$}.
	\end{equation*}
	Here $\dhh(\la)$ is given by (\ref{f1}) and $\mQ_\la^\circ$ 
	is the image in $\Gamma^\circ$ of doubly symmetric $\mQ$-Schur function
	defined in \S\ref{s2.1}.
\end{thm}
\begin{proof}
	This theorem can be proved exactly as Theorem B of the paper 
	\cite{Kerov1998} with the two following changes:
	instead of $\theta$-shifted Jack polynomials one should use 
	factorial Schur's $\mQ$-functions, and 
	instead of \cite[Theorem 6.1]{Kerov1998} 
	one should refer to the relation (\ref{f36}) (proved in the paper by V.~Ivanov \cite{IvanovNewYork3517-3530}).
\end{proof}

To the multiplicative coherent system $\left\{ M_n^{\al} \right\}$ with parameter $\al\in\left( 0,+\infty \right)$
corresponds a measure $\mathsf{P}^{(\al)}$ on $\sk$. We may call it the 
{\em{}multiplicative boundary measure\/}.

\subsection{Doubling of shifted Young diagrams\\ and the Thoma simplex}\label{s8.1}
In this subsection we discuss the embedding
of $\sk$
into the Thoma simplex~$\Omega$
introduced in \cite{GnedinIntern.Math.ResearchNotices2006Art.ID5196839pp.}
(see also \S\ref{s0.2}).

\subsubsection{Modified Frobenius coordinates and the Thoma simplex}
Here we recall some definitions from \cite[\S3.1 and \S3.3]{Borodin2007}.

Let $\si$ be an ordinary partition.\footnote{Ordinary partitions are 
identified with ordinary Young diagrams as in \cite[Ch. I, \S1]{Macdonald1995}}
Denote by $a_1,\dots,a_k,b_1,\dots,b_k$ its 
{\em{}modified Frobenius coordinates\/}. That is, 
$k$ is the number of diagonal boxes in $\si$,
$a_i$ equals $\frac12$ plus the number of boxes in the $i$th row 
to the right of the diagonal, and $b_j$ equals $\frac12$ plus the 
number of boxes in the $j$th column below the diagonal.
We write $\si=\left( a_1,\dots,a_k\mid b_1,\dots,b_k \right)$.
Note that $\sum(a_i+b_i)=|\si|$, the number of boxes in the diagram $\si$.
Note also that each of the sequences $\left\{ a_1,\dots,a_k \right\}$ and
$\left\{ b_1,\dots,b_k \right\}$ is strictly decreasing.
Recall that by $\mathbb{Y}_n$ we denote the set of all ordinary partitions
of weight $n$.

Let $\Omega$
be the {\em{}Thoma simplex\/}, that is, the space of couples
$(\omega;\omega')\in\left[ 0,1 \right]^{\infty}\times\left[ 0,1 \right]^{\infty}$ 
satisfying the following conditions:
\begin{equation*}
	\omega_1\ge\omega_2\ge\dots\ge0,\qquad
	\omega'_1\ge\omega'_2\ge\dots\ge0,\qquad
	\sum_{i}\omega_i+
	\sum_{j}\omega'_j\le1.
\end{equation*}
Here $\left[ 0,1 \right]^{\infty}$ is equipped with the product topology
and hence the space 
$\Omega$ is a compact subset of $\left[ 0,1 \right]^{\infty}\times\left[ 0,1 \right]^{\infty}$.

Consider for all $n\in\N$ embeddings:
\begin{equation*}
	\hat\iota_n\colon\mathbb{Y}_n\hookrightarrow\Omega,\qquad \si\mapsto
	\left( \frac{a_1}n,\dots,\frac{a_k}n,0,\dots;
	       \frac{b_1}n,\dots,\frac{b_k}n,0,\dots\right),
\end{equation*}
where $\si=(a_1,\dots,a_k\mid b_1,\dots,b_k)$ is written in terms of the modified Frobenius
coordinates. 

\subsubsection{Doubling of shifted Young diagrams}
Let $\la=(\la_1,\dots,\la_{\ell})$ be a shifted Young diagram.
By $\Df\la$ denote the {\em{}double\/} of $\la$,
that is, the ordinary Young diagram that is 
written in the modified Frobenius coordinates as
$$
	\Df\la=\left( \la_1+\frac12,\dots,\la_\ell+\frac12\mid\la_1-\frac12,\dots,\la_{\ell}-\frac12 \right),
$$
see \cite[Ch.~I,~\S1, Example~9]{Macdonald1995}.\footnote{By
agreement, $\Df\varnothing=\varnothing$.}
In \cite{Hoffman1992} this object is called the
{\em{}shift-symmetric diagram\/} associated to
a strict partition $\la$.
The number of boxes in $\Df\la$ clearly equals 
twice the number of boxes in $\la$.
In this way we obtain embeddings $\Sb_n\hookrightarrow \mathbb{Y}_{2n}$ for all $n\in\Nn$
(and hence the whole Schur graph is embedded into the Young graph).

\subsubsection{The embedding $T\colon\sk\hookrightarrow\Omega$}
The sets $\hat\iota_n(\mathbb{Y}_n)$ approximate the Thoma simplex $\Omega$
in the same sense as the sets $\iota_n(\Sb_n)$ approximate $\sk$ (see \S\ref{s6.2}). 
Thus, it is natural to consider the following ``limit''
of the embeddings $\Sb_n\hookrightarrow\mathbb{Y}_{2n}$
as $n\to\infty$:
\begin{equation*}
	T\x=(\omega;\omega')=\left( 
	\frac{\x_1}2,\frac{\x_2}2,\dots;
	\frac{\x_1}2,\frac{\x_2}2,\dots\right),\qquad \x=(\x_1,\x_2,\dots)\in\sk.
\end{equation*}
The image of $\sk$ is the whole diagonal subset
$\left\{ (\omega,\omega')\colon\omega=\omega' \right\}$
of $\Omega$. Moreover, $T$ is a homeomorphism
between $\sk$ and this subset. The points $\x\in\sk$ such that $\sum\x_i=1$ map 
to the points $(\omega;\omega)$ such that $\sum(\omega_i+\omega_i)=1$.
This embedding $T$ was introduced
in \cite[\S8.6]{GnedinIntern.Math.ResearchNotices2006Art.ID5196839pp.}.

The property that $T\colon\sk\hookrightarrow\Omega$
is a limit in some sense of the embeddings $\Sb_n\hookrightarrow\mathbb{Y}_{2n}$
may be expressed as follows:
\begin{prop}
	Let $\left\{ \la(n) \right\}$, $n=1,2,\dots$, be a sequence of shifted Young diagrams, 
	$\la(n)\in\Sb_n$, such that, as $n\to\infty$, the points $\iota_n(\la(n))$
	tend to some point $\x\in\sk$.
	Then the points $\hat\iota_{2n}(\Df\la(n))$ tend to $T\x\in\Omega$.
\end{prop}
\begin{proof}
	Clearly, $T\iota_n(\lambda(n)))\to T\x$ as $n\to\infty$.

	For any $\mu=(\mu_1,\dots,\mu_\ell)\in\Sb_n$ we have
	\begin{equation*}
		T\iota_n(\mu)=\left( 
		\frac{\mu_1}{2n},\dots,\frac{\mu_\ell}{2n},0,\dots; 
		\frac{\mu_1}{2n},\dots,\frac{\mu_\ell}{2n},0,\dots\right)
	\end{equation*}
	and 
	\begin{equation*}
		\hat\iota_{2n}(\Df\mu)=\left(\frac{\mu_1+1/2}{2n},\dots,\frac{\mu_\ell+1/2}{2n},0,\dots; 
		\frac{\mu_1-1/2}{2n},\dots,\frac{\mu_\ell-1/2}{2n},0,\dots\right).
	\end{equation*}
	To conclude the proof observe that
	\begin{equation*}
		\Big( \underbrace{\frac1{2n},\dots,\frac1{2n}}_n,0,\dots;
		      \underbrace{\frac1{2n},\dots,\frac1{2n}}_n,0,\dots\Big)\to0,\qquad n\to\infty
	\end{equation*}
	in the topology of $\Omega$.
\end{proof}	
Informally, one can say that the previous Proposition states
\begin{equation*}
	T\iota_n(\la)\approx \hat\iota_{2n}(\Df\la),\qquad \la\in\Sb_n.
\end{equation*}

\subsubsection{Symmetric Thoma's measures}
Here we use the above construction to give another proof of the first claim
of Proposition \ref{p6.6}.

Let us recall the definition 
of the moment coordinates on the Thoma simplex
\cite[\S3.4]{Borodin2007}.
To every point $(\omega;\omega')\in\Omega$ one can assign the following
probability measure on $\left[ -1,1 \right]$:
\begin{equation*}
	\hat\nu_{(\omega;\omega')}:=
	\sum_{i=1}^{\infty}\omega_i\delta_{\omega_i}+
	\sum_{i=1}^{\infty}\omega_i'\delta_{-\omega_i'}+
	\hat\gamma(\omega;\omega')\delta_0,
\end{equation*}
where $\hat\gamma(\omega;\omega')=1-\sum_{i=1}^{\infty}(\omega_i+\omega'_i)$
and $\delta_s$ denotes the Dirac measure at a point~$s$.
This measure is called {\em{}Thoma's measure\/}.
The moments of Thoma's measure $\hat\nu_{(\omega;\omega')}$
are called the moment coordinates of $(\omega;\omega')\in\Omega$:
\begin{equation*}
	\hat\mq_m(\omega;\omega')=\sum_{i=1}^{\infty}\omega_i^{m+1}+
	(-1)^{m}\sum_{j=1}^{\infty}\omega_j'^{m+1},\qquad m=1,2,\dots
\end{equation*}
(compare these
definitions of $\hat\nu_{(\omega;\omega')}$ 
and $\hat\mq_m(\omega;\omega')$ 
to the definitions
of $\nu_\x$ and $\mq_m(\x)$ from~\S\ref{s6.3}).

If $\x\in\sk$, then Thoma's measure $\hat\nu_{T\x}$ is symmetric 
with respect to the origin. Hence the odd moments of $\hat\nu_{T\x}$ vanish.
More precisely,
\begin{equation*}
	\hat\mq_m(T\x)=\left\{
	\begin{array}{ll}
		2^{-m}\mq_m(\x),&\mbox{if $m$ is even};\\
		0,&              \mbox{if $m$ is odd}.
	\end{array}
	\right.
\end{equation*}
Here $\mq_m(\x)$ are the moment coordinates on $\sk$.

A probability measure on $\left[ -1,1 \right]$
is uniquely determined by its moments. Hence the functions $\hat\mq_1,\hat\mq_2,\dots$
separate points of $\Omega$. 
It follows that a point $T\x\in\Omega$ (where $\x\in\sk$) is uniquely determined by its
moment coordinates
$\hat\mq_2(T\x),\hat\mq_4(T\x),\hat\mq_6(T\x),\dots$ (it suffices to
take only even coordinates because the 
odd coordinates vanish). 
This is the same as to say that a point $\x\in\sk$ is uniquely determined
by its even moment coordinates 
$\mq_2(\x),\mq_4(\x),\mq_6(\x),\dots$.
Hence
the first claim of Proposition \ref{p6.6}
holds.

\subsection{Convergence of generators}\label{s6.5}
In this subsection 
we prove the convergence of the operators
$n^2(T_n-{\bf1})$ to 
an operator $A\colon \Fc\to\Fc$. 
In the next section using this convergence we apply the abstract Theorem
\ref{p6.4} to our situation and prove the convergence
of Markov chains
corresponding to $n^2(T_n-{\bf1})$ to the Markov process $\X\al(t)$ in $\sk$.
\begin{prop}\label{p6.8}
	In the sense of Definition \ref{p6.1} we have
	\begin{equation*}
		n^2(T_n-\mathbf{1})\pi_nf\to Af\qquad\mbox{for all $f\in\Fc$},
	\end{equation*}
	where the operator $A\colon\Fc\to\Fc$ can be written in one of the two following ways:

	{\rm{}(1)} As a formal differential operator\footnote{See also Remark
	\ref{p40.2} about formal differential operators in polynomial algebras.}
	in the algebra $\Fc=\mathbb{R}\left[ \mq_2,\mq_4,\mq_6,\dots \right]$:
	\begin{equation}\label{f85}
		\left.
		\begin{array}{l}
			\displaystyle
			A=\sum_{i,j=1}^{\infty}(2i+1)(2j+1)\left( \mq_{2i+2j}-\mq_{2i}\mq_{2j} \right)
			\frac{\partial^2}{\partial\mq_{2i}\partial\mq_{2j}}\\\displaystyle\qquad+
			2\sum_{i,j=0}^{\infty}\left( 2i+2j+3 \right)\mq_{2i}\mq_{2j}\frac{\partial}{\partial \mq_{2i+2j+2}}
			-\sum_{i=1}^{\infty}(2i+1)\left( 2i+\frac\al2 \right)
			\mq_{2i}\frac{\partial}{\partial\mq_{2i}},
		\end{array}
		\right.
	\end{equation}
	where, by agreement, $\mq_0:=1$; 

	{\rm{}(2)} 
	As an operator acting on functions $\mQ_\mu^\circ\in\Fc\subset C(\sk)$:\footnote{Observe that 
	the functions $\mQ_\mu^\circ\in\Gamma^\circ\cong\Fc$, $\mu\in\Sb$,
	are not linearly independent. However, their linear span is $\Fc$ because 
	the system $\left\{ \mQ_\mu \right\}_{\mu\in\Sb}$ is a basis for $\Gamma$.
	However, from the claim (1)
	it follows that the formula (\ref{f86})
	for 
	the action of $A$ on $\mQ_\mu^\circ$, $\mu\in\Sb$,
	is consistent.}
	\begin{equation}\label{f86}
			\displaystyle
			A\mQ_\mu^\circ=-|\mu|(|\mu|+\al/2-1)\mQ_\mu^\circ+
			\sum_{y\in Y(\mu)}\big(y(y+1)+\al\big)\mQ_{\mu-\square(y)}^\circ;
	\end{equation}
\end{prop}

First, we prove two Lemmas.
\begin{lemma}\label{p6.9}
	Let $\phi\in\Gamma$ and $\deg\phi\le m-1$ for some $m\in\N$. Then
	\begin{equation*}
		\frac1{n^m}\phi_n\to0,\qquad n\to\infty
	\end{equation*}
	in the sense of Definition \ref{p6.1}.\footnote{Recall that
	$(\cdots)_n$ denotes the restriction of a function from the algebra
	$\Gamma\subset\mathrm{Fun}(\Sb)$ 
	to the subset $\Sb_n\subset\Sb$.}
\end{lemma}
\begin{proof}
	The convergence to zero means that
	\begin{equation*}
		\sup_{\la\in\Sb_n}\frac1{n^m}\left|\phi_n(\la)\right|\to0,\qquad n\to\infty.
	\end{equation*}
	Observe that 
	for all $\la\in\Sb_n$
	we have
	$\la_i\le n$, $i=1,\dots,\ell(\la)$.
	Hence $\left|\phi_n(\la)\right|\le \mathrm{Const}\cdot n^{m-1}$.
\end{proof}
Let $G$ denote the operator $\sum_{i=1}^{\infty}(2i-1)p_{2i-1}\frac{\partial}{\partial p_{2i-1}}$ in the algebra $\Gamma$.
In other words,
on the homogeneous component of degree $m$ of the algebra $\Gamma$, $m\in\Nn$, 
the operator $G$ acts as multiplication by $m$.
Let for all $s>0$ the operator $s^G\colon\Gamma\to\Gamma$ be the automorphism of $\Gamma$ which reduces to
multiplication by $s^m$ on the $m$th homogeneous component of $\Gamma$.

Recall the maps $\pi_n\colon C(\sk)\to\mathrm{Fun}(\Sb_n)$ defined in \S\ref{s6.2}.
\begin{lemma}\label{p6.10}
	Let $g\in\Gamma$ and $f=g^\circ\in\Fc$. For all $n\in\N$ we have
	\begin{equation*}
		\pi_nf=\left( n^{-G}g \right)_n.
	\end{equation*}
\end{lemma}
\begin{proof}
	Fix $\la\in\Sb_n$.
	Consider the homomorphism $\Gamma\to\mathbb{R}$ 
	defined on the generators as follows:
	\begin{equation*}
		p_{2m-1}\to\frac1{n^{2m-1}}{\sum_{i=1}^{\ell(\la)}\la_i^{2m-1}},\quad m=1,2,3,\dots.
	\end{equation*}
	On one hand, this homomorphism is a composition of the automorphism $n^{-G}$ of $\Gamma$
	and the map $\Gamma\to\mathbb{R}$, $\phi\mapsto \phi_n(\la)$, hence 
	$g\in\Gamma$ maps to $(n^{-G}g)_n(\la)$.
	On the other hand, since $p_1$ maps to $1$, this homomorphism can be viewed as a composition
	of the canonical map $\Gamma\to\Gamma^\circ$ and the map $\Gamma^\circ\to\mathbb{R}$, $\psi\mapsto\psi(\iota_n(\la))$
	(here $\iota_n$ is defined in \S\ref{s6.2}).
	Hence $g\in\Gamma$ maps to $f(\iota_n(\la))$.
	This concludes the proof.
\end{proof}

\par\noindent{\em{}Proof of Proposition \ref{p6.8}.\/} 
Fix arbitrary $f\in\Fc$. 
Let $g\in\Gamma$ be such that $f=g^\circ$.
Theorem \ref{p5.1} and Lemma \ref{p6.10} imply
\begin{equation}\label{f90}
	n^2(T_n-{\bf1})\pi_n(f)=
	n^2(T_n-{\bf1})\left( n^{-G}g \right)_n=
	\frac{n^2}{(n+\al/2)(n+1)}\big(\widetilde Bn^{-G}g\big)_n,
\end{equation}
where $\widetilde B$ is a zero
degree operator in the algebra $\Gamma$ 
with the top degree homogeneous part $B\colon\Gamma\to\Gamma$
given by (\ref{f75}).

Because $B-\widetilde B$ has degree $-1$, we can 
replace $\widetilde B$ by $B$ in (\ref{f90}), 
this will affect only negligible terms.\footnote{One can argue as follows.
Without loss of generality, assume that $g$ is homogeneous of degree $m\in\N$.
Thus, $n^{-G}g=n^{-m}g$. Moreover, $\deg(B-\widetilde B)g\le m-1$ and hence
by Lemma \ref{p6.9},
$\big( (B-\widetilde B)n^{-G}g\big)_n=\frac1{n^m}\big( (B-\widetilde B)g\big)_n\to 0$, $n\to\infty$.}
We can also remove the
factor $\frac{n^2}{(n+\al/2)(n+1)}$.
Thus, we have
\begin{equation*}
	n^2(T_n-{\bf1})\pi_n(f)-(Bn^{-G}g)_n\to0,\qquad n\to\infty.
\end{equation*}
The operator $B\colon \Gamma\to\Gamma$ is homogeneous, therefore, $Bn^{-G}=n^{-G}B$, and
by Lemma \ref{p6.10} we have
\begin{equation*}
	\begin{array}{l}\displaystyle
		(Bn^{-G}g)_n=(n^{-G}Bg)_n=
		\pi_n\big( (Bg)^{\circ} \big)
	\end{array}
\end{equation*}

Recall that the operator $B\colon\Gamma\to\Gamma$ commutes with the multiplication by $p_1$
(Corollary \ref{p5.2}), therefore,
it induces an operator $A\colon\Fc\to\Fc$,
$f\mapsto (Bg)^\circ$, where $g\in\Gamma$ is such that $f=g^{\circ}$.
Clearly, $Af$ does not depend on the choice of $g$.

Since $\Fc\cong\mathbb{R}\left[ p_3,p_5,\dots \right]$, 
we get (\ref{f85}) from (\ref{f75})
by replacing $p_1$ by $1$ and each $p_{2m+1}$ by $\mq_{2m}$, $m\in\N$.

It remains to prove (\ref{f86}).
Fix $\mu\in\Sb$. Multiply (\ref{f32}) by $n^{-|\mu|}n^2$:
\begin{equation*}
	\begin{array}{l}\displaystyle
		n^{-|\mu|}n^2(T_n-1)(\mQ_\mu^*)_n\\\displaystyle\qquad=
		\frac{n^2}{(n+\al/2)(n+1)}
		\Bigg[-|\mu|(|\mu|+\al/2-1)n^{-|\mu|}(\mQ_\mu^*)_n\phantom{\Bigg].}\\\displaystyle\qquad\qquad
		+\frac{n-|\mu|+1}{n}\sum_{y\in Y(\mu)}\big(y(y+1)+\al\big)n^{-|\mu|+1}(\mQ_{\mu-\square(y)}^*)_n
		\Bigg].
	\end{array}
\end{equation*}
Since $\deg(\mQ_{\la}^*-\mQ_\la)\le|\la|-1$ (see \S\ref{s2.1}), each
function of the form $(\mQ_\la^*)_n$ in the RHS can be replaced by $(\mQ_\la)_n$,
this affects only negligible terms. We can also remove fractions containing $n$. Thus,
\begin{equation*}
	\begin{array}{l}\displaystyle
		\Big(n^{-|\mu|}n^2(T_n-1)(\mQ_\mu)_n+
		|\mu|(|\mu|+\al/2-1)n^{-|\mu|}(\mQ_\mu)_n
		\\\qquad\qquad\qquad\displaystyle
		-\sum_{y\in Y(\mu)}\big(y(y+1)+\al\big)n^{-|\mu|+1}(\mQ_{\mu-\square(y)})_n
		\Big)\to0,\qquad n\to\infty
	\end{array}
\end{equation*}
in the sense of Definition \ref{p6.1}.
To conclude the proof of Proposition \ref{p6.8}
observe that $n^{-|\la|}(\mQ_\la)_n=\left( n^{-G}\mQ_\la \right)_n=\pi_n(\mQ_\la^\circ)$.
\qed

\subsection{Convergence of semigroups\\ and the existence of the process}\label{s6.6}
To apply Theorem \ref{p6.4} to our situation 
and finally get 
the existence of the process $\X\al(t)$ in $\sk$
it remains to prove the following:
\begin{lemma}\label{p6.11}
	{\rm{}(1)\/} 
	The operator $A\colon\Fc\to\Fc$ from Proposition \ref{p6.8} is dissipative. 
	
	{\rm{}(2)\/} For all $s>0$, the range of $s{\bf1}-A$ is dense in $C(\sk)$.
\end{lemma}
\begin{proof} (cf. \cite[Proof of Proposition 1.4]{Borodin2007}) 
	(1)
	Consider the filtration of the algebra $\Fc=\Gamma/(p_1-1)\Gamma$ inherited from 
	the natural filtration (\ref{f45}) of
	$\Gamma$:
	\begin{equation*}
		\Fc=\bigcup_{m=0}^{\infty}\Fc^{m},\qquad \Fc^0\subset\Fc^1\subset\Fc^2\subset\dots\subset\Fc.
	\end{equation*}
	It is clear that for fixed $m$
	the operator $\pi_n\colon C(\sk)\to\mathrm{Fun}(\Sb_n)$ 
	is injective on $\Fc^m$ for all $n$ large enough (this is true because each $\Fc^m$ is finite-dimensional
	and the spaces $\iota_n(\Sb_n)$ approximate the space $\sk$, see \S\ref{s6.2}).
	Thus, we can identify $\Fc^m$ with $\pi_n(\Fc^m)$ for such $n$.

	It follows from (\ref{f32}) that 
	the operator $T_n$ does not increase the degree of functions.
	Therefore, 
	we can think of $T_n$ (and, clearly, of $n^2(T_n-{\bf1})$)
	as an operator in $\Fc^m$. 
	The convergence $n^2(T_n-\mathbf{1})\to A$ 
	established in Proposition~\ref{p6.8} implies that $n^2(T_n-\mathbf{1})$ converges
	to $A$ in every finite-dimensional space $\Fc^m$, $m\in\Nn$.

	Fix $m\in\Nn$.
	For all $n$ large enough the operator 
	$n^2(T_n-\mathbf{1})$ (viewed as an operator in $\Fc^m$)
	is dissipative with respect to the norm of $\mathrm{Fun}(\Sb_n)$
	because the operator $T_n$ is a transition operator of a Markov chain.
	Since the norms of $\mathrm{Fun}(\Sb_n)$ converge to the norm of $C(\sk)$ 
	(in the sense of (\ref{f80})),
	we conclude that $A$ is dissipative.

	(2) For every $m$ and $s>0$ the operator $s{\bf1}-A$ maps $\Fc^m$ onto itself
	(this fact can be derived either from (\ref{f86}) 
	or from the above proof of the claim (1) of the present Lemma).
	Thus, $(s{\bf1}-A)\Fc=\Fc$ and therefore $s{\bf1}-A$ has a dense range.
\end{proof}
Now, from
Theorem \ref{p6.4}
it follows that the operator $A$ (given by Proposition \ref{p6.8})
is closable in $C(\sk)$ 
and its closure generates a strongly 
continuous contraction semigroup $\left\{ T(t) \right\}_{t\ge0}$.

We also have the convergence of semigroups $\left\{ \mathbf{1},T_n,T_n^2,\dots \right\}$
to $\left\{ T(t) \right\}$ (\ref{f81}).
Hence the 
semigroup $\left\{ T(t) \right\}$ 
preserves the cone of nonnegative functions and the
constant
function $1$ because each $T_n$ possesses this property.

From \cite[Chapter 4, Theorem 2.7]{Ethier1986} follows that the semigroup $\left\{ T(t) \right\}$
gives rise to a strong Markov process $\X{\al}(t)$ in $\sk$. This process 
has c\`adl\`ag sample paths and can start from any point and any probability distribution on $\sk$.

The operator $A$ is called the {\em{}pre-generator\/} of the process $\X{\al}(t)$, $t\ge0$.

\subsection{Some properties of the process $\X{\al}(t)$}\label{s70.7}
Here we state some properties of the process $\X{\al}(t)$.
They follow from the construction of 
$\X{\al}(t)$ and from the formulas (\ref{f85}) and (\ref{f86})
for its pre-generator.
We do not give the proofs because they are similar to 
those from the paper \cite{Borodin2007}.
\begin{rmk}\rm{}\label{p90.1}
	The formula 
	(\ref{f86})	
	for the action of the pre-generator
	$A$ on functions $\mQ_\mu^\circ\in\Fc\subset C(\sk)$, $\mu\in\Sb$,
	is not formally necessary 
	for the convergence of the up/down Markov chains from Definition \ref{p10.6}
	to a continuous time Markov process in $\sk$, 
	as well as for the properties
	of the limit diffusion $\X\al(t)$ that are listed 
	in this subsection.\footnote{However,
	(\ref{f86}) allows to argue in a more straightforward way at some 
	points. It can be also interesting to compare this formula 
	with similar ones
	in other models
	(namely, \cite[(5.1)]{Borodin2007} and \cite[(14)]{Petrov2007}).}
	Indeed, from (\ref{f85}) one can easily obtain that 
	\begin{equation}\label{f8080}
		Af=-m(m-1+\al/2)f+g,\qquad\mbox{where $g\in\Fc^{m-1}$}
	\end{equation}
	for all $f\in \Fc^m$, $m\in\Nn$. 
	In other words, the action of $A$ on $\Fc^m$ 
	up to lower degree terms (that is, terms from $\Fc^{m-1}$)
	is the multiplication
	by $-m(m-1+\al/2)$.
	It suffices to check (\ref{f8080}) for $f=p_{\si_1}^\circ\dots p_{\si_\ell}^\circ$, 
	where $p_i$'s are the Newton power sums and $\si=(\si_1,\dots,\si_\ell)$ is an odd partition
	without parts equal to one. Indeed, for each $m\in\Nn$
	the functions of this form 
	with $|\si|\le m$
	constitute a basis for $\Fc^m$.
	For such $f$ 
	the relation (\ref{f8080})
	can be easily checked directly
	using (\ref{f85}) 
	(note that $p^{\circ}_{2i+1}=\mq_{2i}$, $i\ge1$)
\end{rmk}

\medskip
{\bf{}Continuity of sample paths.\/}
{\em{}The process $\X{\al}(t)$ has continuous sample paths.\/}

\smallskip

This is proved exactly as in 
\cite[Coroll. 6.4 and Thm. 7.1]{Borodin2007}, the proof uses the expression (\ref{f85}) for $A$.

\medskip

{\bf{}The invariant symmetrizing measure.\/}
{\em{}The multiplicative boundary measure $\mathsf{P}^{(\al)}$ defined in Theorem \ref{p6.7}
is an invariant measure for $\X{\al}(t)$. The process is reversible with
respect to $\mathsf{P}^{(\al)}$.\/}

\smallskip

This follows from the facts that 

$\bullet$ For all $n\in\N$ the measure $M_n^{\al}$ is an 
invariant symmetrizing distribution for the Markov chain $T_n$ (see \S\ref{s1});

$\bullet$ The measures $M_n^\al$ approximate the measure $\mathsf{P}^{(\al)}$ in the sense of Theorem \ref{p6.7};

$\bullet$ The chains $T_n$ approximate the process $\X{\al}(t)$ (see \S\ref{p6.6}).

See also \cite[Prop. 1.6, 1.7, Thm. 7.3 (2)]{Borodin2007}.

\medskip

{\bf{}Convergence of finite-dimensional distributions\/}
(Cf. \cite[Prop. 1.8]{Borodin2007}).
{\em{}Let $\X{\al}(t)$ and all the chains $T_n$ are viewed in equilibrium (that is, starting
from the invariant distribution). Then the finite-dimensional distributions for the
$n$th chain converge, as $n\to\infty$, to the corresponding finite-dimensional distributions
of the process $\X{\al}(t)$. Here we assume a natural scaling of time: one step of $n$th
Markov chain $T_n$ corresponds to a small time interval $\Delta t\sim 1/n^2$.\/}

\medskip

{\bf{}The spectrum of the Markov generator in $L^2(\sk,\mathsf{P}^{(\al)})$.\/}
{\em{}The space $\Fc$ 
viewed as the subspace of the Hilbert space
$L^2(\sk,\mathsf{P}^{(\al)})$
is decomposed into the orthogonal
direct sum of eigenspaces of the operator $A\colon\Fc\to\Fc$.
The eigenvalues of $A$ are 
\begin{equation}\label{f99}
	\left\{ 0 \right\}\cup\left\{ -m\left( m-1+\frac\al2 \right) \colon m=2,3,\dots\right\}.
\end{equation}
The eigenvalue $0$ is simple, and the multiplicity of each $m$th eigenvalue 
is equal to the number of odd
partitions of $m$ without parts equal to one, that is, to the number of solutions of the equation
\begin{equation*}
	3n_3+5n_5+7n_7+\ldots=m
\end{equation*}
in nonnegative integers.\/}

\smallskip

The reversibility property of the process $\X{\al}(t)$ implies that
the operator
$A$ is symmetric with respect to the inner product inherited from $L^2(\sk,\mathsf{P}^{(\al)})$
Moreover, $A$
preserves the filtration of $\Fc$ (defined in \S\ref{s6.6}). 
Indeed, this follows from the expression (\ref{f86}) for the pre-generator.

The fact that the eigenvalues of $A$ are described by (\ref{f99}) follows from (\ref{f86}). Indeed,
$\mQ_\mu^\circ\in\Fc^{|\mu|}$ for all $\mu\in\Sb$,
and (\ref{f86}) is rewritten as
$A\mQ_\mu^\circ=-|\mu|(|\mu|-1+\al/2)\mQ_\mu^\circ+g$, where $g\in\Fc^{|\mu|-1}$.
Thus, the eigenvalue $0$ is simple, and the multiplicity of each $-m(m-1+\al/2)$ is 
equal to $(\dim\Fc^m-\dim\Fc^{m-1})$,
$m\in\N$.

Since $\Fc\cong\mathbb{R}\left[ p_3,p_5,p_7,\dots \right]$, finite products of the form 
$p_3^{r_3}p_5^{r_5}p_7^{r_7}\dots$ constitute a linear basis for $\Fc$. This basis is compatible
with the filtration $\left\{ \Fc^m \right\}$.
Hence $(\dim\Fc^m-\dim\Fc^{m-1})$ is equal to 
the number of basis vectors of degree $m$, $m\in\N$, which is exactly the number
of odd partitions of $m$ without parts equal to one.

\medskip

{\bf{}The uniqueness of the invariant measure\/} 
(Cf. \cite[Thm. 7.3 (1)]{Borodin2007}).
{\em{}The measure $\mathsf{P}^{(\al)}$
is a unique invariant measure for the process $\X{\al}(t)$.\/}

\medskip

{\bf{}Ergodicity.\/}
{\em{}The process $\X{\al}(t)$ is ergodic with respect to the measure $\mathsf{P}^{(\al)}$.\/}

\smallskip

This follows from the existence of a spectral gap of the process’ generator, see
the eigenstructure above. See also \cite[Thm. 7.3 (3)]{Borodin2007}.

\bigskip

Dobrushin Mathematics Laboratory, 
Kharkevich Institute for Information Transmission Problems,
Bolshoy Karetny 19, 127994 Moscow GSP-4, Russia.

\medskip

E-mail: \texttt{lenia.petrov@gmail.com}
\end{document}